\documentclass[a4paper,abstracton]{scrartcl}

\usepackage[utf8]{inputenc}

\usepackage{amssymb} % set of additional math symbols
\usepackage{amsmath}  % new environments
\usepackage{amsthm}
\usepackage{todonotes}
\usepackage{hyperref}
\usepackage[ruled,vlined,linesnumbered]{algorithm2e}

\usepackage{graphicx}
\usepackage{subfigure}
\usepackage{verbatim,tabularx,enumitem,colonequals,psfrag}
\usepackage{cite}
\usepackage{authblk}
\usepackage[normalem]{ulem} 
\usepackage{framed}
\usepackage{rotating}

\usepackage{epstopdf}
\usepackage{tikz}
\usetikzlibrary{patterns}

\usepackage[scr=boondoxo]{mathalfa}

\newcommand{\cU}{\mathcal{U}}

\newcommand{\cJ}{\mathcal{J}}

\newcommand{\yl}{\mathscr{y}}
\newcommand{\yu}{\overline{\mathscr{y}}}
\newcommand{\ysl}{z}
\newcommand{\ysu}{\overline{z}}

\newcommand{\cX}{\mathscr{X}}
\newcommand{\coX}{\overline{\mathscr{X}}}
\newcommand{\cY}{\mathscr{Z}}
\newcommand{\cZ}{\overline{\mathscr{Z}}}

\newcommand{\Rset}{\mathbb{R}}

\newcommand{\cl}{\underline{c}}
\newcommand{\cu}{\overline{c}}
\newcommand{\cd}{d}

\newtheorem{theorem}{Theorem}
\newtheorem{lemma}[theorem]{Lemma}

\newtheorem{corollary}[theorem]{Corollary}
\newtheorem{proposition}[theorem]{Proposition}

\newtheorem{observation}[theorem]{Observation}

\newcommand{\BIGOP}[1]{\mathop{\mathchoice%
{\raise-0.22em\hbox{\huge $#1$}}%
{\raise-0.05em\hbox{\Large $#1$}}{\hbox{\large $#1$}}{#1}}}

\allowdisplaybreaks

\begin{document}

\title{On   Recoverable  and Two-Stage Robust Selection Problems with Budgeted Uncertainty}

\author[1]{Andr\'e Chassein}
\author[2]{Marc Goerigk}
\author[3]{Adam Kasperski}
\author[4]{Pawe{\l} Zieli\'nski}

\affil[1]{Fachbereich Mathematik, Technische Universit\"at Kaiserslautern, Germany}
\affil[2]{Department of Management Science, Lancaster University, United Kingdom}
\affil[3]{Faculty of Computer Science and Management, Wroc{\l}aw University of Technology, Poland}
\affil[4]{Faculty of Fundamental Problems of Technology, Wroc{\l}aw University of Technology, Poland}

%\date{}

\maketitle

% \todo[author=Pawel, size=\footnotesize]{There are too many word ``consider'' in the whole paper} 

\begin{abstract}
In this paper the problem of selecting $p$ out of $n$ available items is discussed, such that their total cost is minimized. We assume that costs are not known exactly, but stem from a set of possible outcomes.
% \todo[author=Adam, size=\footnotesize]{total costs or total cost?} 

Robust  recoverable and two-stage models of this selection problem are analyzed. In the two-stage problem, up to $p$ items is chosen in the first stage, and the solution is completed once the scenario becomes revealed in the second stage. In the recoverable problem, a set of $p$ items is selected in the first stage, and can be modified by exchanging up to $k$ items in the second stage, after a scenario reveals.

We assume that uncertain costs are modeled through bounded uncertainty sets, i.e., the interval uncertainty sets with an additional linear (budget) constraint, in their discrete and continuous variants. Polynomial algorithms for  recoverable and two-stage selection problems with continuous bounded uncertainty, and compact mixed integer formulations in the case of discrete bounded uncertainty are constructed.
\end{abstract}


\section{Introduction}

In this paper we consider the following \textsc{Selection} problem.  We are given a set of $n$ items with cost~$c_i$ for each $i\in[n] := \{1,\dots,n\}$ and an integer $p\in [n]$. We seek a subset $X\subseteq [n]$  of $p$ items, $|X|=p$, whose total cost $\sum_{i\in X} c_i$ is minimum. It is easy to see that an optimal solution is composed of $p$ items of the smallest cost. It can be found in $O(n)$ time by using the well-known fact, that the $p$th smallest item can be found in $O(n)$ time (see, e.g.,~\cite{CO90}). \textsc{Selection} is a basic resource allocation problem~\cite{IK88}. It is also a special case of 0-1 knapsack, 0-1 assignment, single machine scheduling, and minimum matroid base problems (see~\cite{KKZ13} for an overview). It can be formulated as the following integer linear program:
\begin{equation}
\label{selproblem}
\begin{array}{lll}
\min\ & \sum_{i\in[n]} c_i x_i \\
\text{s.t. } & \sum_{i\in[n]} x_i = p \\
& x_i\in\{0,1\} & \forall i\in[n].
\end{array}
\end{equation}
 We will use $\Phi\subseteq \{0,1\}^n$ to denote the set of all feasible solutions to~(\ref{selproblem}). Given 
 $\pmb{x}\in \{0,1\}^n$, we also define $X_{\pmb{x}}=\{i\in [n]: x_i=1\}$, and $\overline{X}_{\pmb{x}}=[n]\setminus X_{\pmb{x}}$, i.e. $X_{\pmb{x}}$ is the item set induced by vector $\pmb{x}$ and $\overline{X}_{\pmb{x}}$ denotes its complement.

Consider the case when the item costs are uncertain. As  part of the input, we are given a scenario set $\mathcal{U}$, containing all possible vectors of the item costs, called \emph{scenarios}. Several methods of defining $\mathcal{U}$ have been proposed in the existing literature (see, e.g.,~\cite{BS03, BS04, K08, KY97,NO13}). 
Under \emph{discrete uncertainty} (see, e.g.,~\cite{KY97}), the scenario set contains $K$ distinct scenarios i.e. $\mathcal{U}^D=\{\pmb{c}_1,\dots,\pmb{c}_K\}$,
$\pmb{c}_i \in\Rset^n_{+}$. Under \emph{interval uncertainty}, the cost of each item $i\in [n]$ belongs to a closed interval $[\cl_i, \cu_i]$, where $\cd_i:=\cu_i - \cl_i \geq 0$ is the maximal deviation of the cost of~$i$ from its 
\emph{nominal value}~$\cl_i$. In the traditional interval uncertainty representation, $\mathcal{U}^I$ is the Cartesian product of all the intervals (see, e.g.,~\cite{KY97}). In this paper we will focus on the following two generalizations of  scenario set $\mathcal{U}^I$, which have been examined in~\cite{BS03, BS04, NO13}:
\begin{itemize}
\item \emph{Continuous budgeted uncertainty}:
\[ \cU^c = \{ (\cl_i+\delta_i)_{i\in [n]} : \delta_i \in [0, \cd_i], \sum_{i\in[n]} \delta_i \le \Gamma\} \subseteq 
\Rset^n_{+}\] 
\item \emph{Discrete budgeted uncertainty}:
\[ \cU^d = \{ (\cl_i+\delta_i)_{i\in [n]} : \delta_i \in \{0,\cd_i\}, |\{i\in[n] : \delta_i = \cd_i\}| \le \Gamma \} \subseteq \Rset^n_{+} \]
\end{itemize}
The fixed parameter $\Gamma\geq 0$ is called a \emph{budget} and it controls the amount of uncertainty which an adversary can allocate to the item costs. For a sufficiently large $\Gamma$, $\cU^c$ reduces to $\mathcal{U}^I$, and $\cU^d$ reduces to the extreme points of $\cU^I$.

In order to compute a solution, under a specified scenario set $\mathcal{U}$, we can follow a robust optimization approach. For general overviews on robust optimization, see, e.g., \cite{ABV09,GMT14,KY97,KZ16b,R10}. In a typical, single-stage robust model we seek a solution minimizing the total cost in a worst case. This leads to the 
following \emph{minmax} and \emph{minmax} regret problems:
$$\textsc{MinMax}:\; \min_{\pmb{x}\in \Phi}\max_{\pmb{c}\in \cU} \pmb{c}\pmb{x},$$
$$\textsc{MinMax-Regret}:\; \min_{\pmb{x}\in \Phi}\max_{\pmb{c}\in \cU} \max_{\pmb{y}\in \Phi} (\pmb{c}\pmb{x}-\pmb{c}\pmb{y}).$$
The minmax (regret) versions of the \textsc{Selection} problem have been discussed in the existing literature. For scenario set $\mathcal{U}^D$ both problems are NP-hard even for 
$K=2$ (the number of scenarios equals~2)~\cite{A01}. If $K$ is part of the input, then \textsc{MinMax} and \textsc{MinMax-Regret} are 
strongly NP-hard and
not approximable within any constant factor~\cite{KKZ13}. On the other hand, \textsc{MinMax} is approximable within $O(\log K/ \log\log K)$~\cite{D13} but \textsc{MinMax-Regret} is only known to be approximable within~$K$,
which is due to the results given in~\cite{ABV09}.
 The \textsc{MinMax} problem under scenario sets $\mathcal{U}^c$ and $\mathcal{U}^d$ is polynomially solvable, according to the results obtained in~\cite{BS03}. Also, \textsc{MinMax-Regret}, under scenario set $\mathcal{U}^I$, is polynomially solvable by the algorithms designed in~\cite{A01, C04}.

The problems which arises in practice often have a two-stage nature. Namely, a partial solution is computed in the first stage and completed in the second stage, or a complete solution is formed in the first stage and  modified  to some extent in the second stage. Typically, the costs in the first stage are precisely known, while the costs in the second stage are uncertain.  Before we formally define the two-stage models, let us introduce some additional notation:
\begin{itemize}
\item $\Phi_1 = \{ \pmb{x}\in\{0,1\}^n : \sum_{i\in[n]} x_i \le p \}$,
\item $\Phi_{\pmb{x}} = \{ \pmb{y}\in\{0,1\}^n : \sum_{i\in[n]} (x_i + y_i) = p, x_i+y_i\le 1, i\in [n] \}, 
\;\;\pmb{x}\in \Phi_1$,
\item $\Phi^k_{\pmb{x}} = \{ \pmb{y}\in\{0,1\}^n : \sum_{i\in [n]} y_i = p,\ \sum_{i\in [n]} x_iy_i \ge p-k \},
\;\; \pmb{x}\in \Phi, k\in [p]\cup\{0\}.$
\end{itemize}
If $\pmb{y}\in \Phi_{\pmb{x}}$, then $X_{\pmb{x}}\cap X_{\pmb{y}}=\emptyset$ and $|X_{\pmb{x}}\cup X_{\pmb{y}}|=p$. Hence $\Phi_{\pmb{x}}$ encodes all subsets of  the item set $[n]$, which added to $X_{\pmb{x}}$ form a complete solution of cardinality~$p$. Set $\Phi^k_{\pmb{x}}$ is called a \emph{recovery set},
$k$ is a given  \emph{recovery parameter}. If $\pmb{y}\in \Phi^k_{\pmb{x}}$, then $|X_{\pmb{x}}\setminus X_{\pmb{y}}|=|X_{\pmb{y}} \setminus X_{\pmb{x}}|\leq k$, so $\Phi^k_{\pmb{x}}$ encodes all solutions which can be obtained from $X_{\pmb{x}}$ by exchanging up to $k$ items. Let $\pmb{C}=(C_1,\dots, C_n)$ be a vector of the first stage item costs, which are assumed to be precisely known. Let scenario set $\cU$ contain all possible vectors of the uncertain second stage costs. Given $k\in [p]\cup\{0\}$, we study the following recoverable selection problem:
$$ \textsc{RREC}: \min_{\pmb{x}\in\Phi} \left( \pmb{C}\pmb{x}+ \max_{\pmb{c}\in\cU} \min_{\pmb{y}\in \Phi^k_{\pmb{x}}} \pmb{c}\pmb{y} \right).$$
In \textsc{RREC} a complete solution (exactly $p$ items) is chosen in the first stage. Then, after a scenario from $\cU$ reveals, one can exchange optimally up to $k$ items in the second stage. Notice that if $k=0$ and $C_i=0$ for each $i\in [n]$, then \textsc{RREC} becomes the \textsc{MinMax} problem.
The robust recoverable model for linear programming, together with some applications, was discussed in~\cite{LLMS09}. It has been also recently applied to the shortest path~\cite{B12}, spanning tree~\cite{HKZ16,HKZ16a}, knapsack~\cite{BKK11} and traveling salesman problems~\cite{CG15b}. The \textsc{RREC} problem under scenario sets $\cU^D$ and $\cU^I$ has been recently discussed in~\cite{KZ15b}. Under $\cU^D$ it turned out to be NP-hard for constant $K$, strongly NP-hard and not at all approximable when $K$ is  part of the input (this is true even if $k=1$). On the other hand, under scenario set $\cU^I$, a polynomial $O((p-k)n^2)$ time algorithm for \textsc{RREC} has been proposed in~\cite{KZ15b}. No results for scenario rests $\cU^c$ and $\cU^d$ have been known to date.

We also analyze the following robust two-stage selection problem:
$$\textsc{R2ST}:\; \min_{\pmb{x}\in\Phi_1} \left(\pmb{C}\pmb{x}+\max_{\pmb{c}\in\cU} \min_{\pmb{y}\in \Phi_{\pmb{x}}} \pmb{c}\pmb{y}\right),$$
In \textsc{R2ST} we seek a first stage solution, which may contain less than $p$ items. Then, after a scenario from $\cU$ reveals, this solution is completed optimally to $p$ items. The robust two-stage model was introduced in~\cite{KMU08} for the bipartite matching problem. The R2ST problem has been recently discussed in~\cite{KZ15b}. It is polynomially solvable under scenario set $\cU^I$. For scenario set $\cU^D$, the problem is strongly NP-hard and hard to approximate within $(1-\epsilon)\log n$ for any $\epsilon>0$,  but it has an $O(\log K + \log n)$ randomized approximation algorithm. No results for scenario sets $\cU^c$ and $\cU^d$ have been known to date.

Given a first stage solution $\pmb{x}\in \Phi$ (resp. $\pmb{x}\in \Phi_1$), we will also study the following \emph{adversarial problem}: 
$$\textsc{AREC (A2ST)}: \;\max_{\pmb{c}\in\cU} \min_{\pmb{y}\in \Phi^k_{\pmb{x}}(\Phi_{\pmb{x}})} \pmb{c}\pmb{y}.$$
If, additionally, scenario $\pmb{c}\in \cU$ is fixed, then we get the following \emph{incremental problem}:
$$\textsc{IREC (I2ST)}: \;\min_{\pmb{y}\in \Phi^k_{\pmb{x}}(\Phi_{\pmb{x}})} \pmb{c}\pmb{y}.$$
The adversarial and incremental versions of some network problems were discussed in~\cite{NO13, SAO09}. The incremental versions of the shortest path and the spanning tree problems are polynomially solvable~\cite{SAO09}, whereas the adversarial versions of these problems under scenario set $\cU^d$ are strongly NP-hard~\cite{NO13, FR99, LC93}.

\begin{table}[ht]
\centering
\caption{The known results for $\cU^D$ and $\cU^I$ obtained in~\cite{KZ15b} and new results for $\cU^c$ and $\cU^d$ shown in this paper.} \label{tab1}
\begin{small}
\begin{tabular}{l|lll|lllllll}
$\cU$	& \textsc{IREC} & \textsc{AREC} & \textsc{RREC} & \textsc{I2ST} & \textsc{A2ST} & \textsc{R2ST} \\ \hline
 & $O(n)$ & $O(Kn)$ & NP-hard for const. $K$; & $O(n)$ & $O(Kn)$ & NP-hard for const. $K$; \\
		 & & & str. NP-hard					&		&	&  str. NP-hard  \\
$\cU^D$  			       &	      &                & not at all appr. &	&	&	 appr. $O(\log K + \log n)$\\
			       &	     &			&	for unbounded $K$		&	&	&  not appr. $(1-\epsilon)\log n$ \\
			       &		&		&						&	&      & for unbounded $K$ \\ \hline 
$\cU^I$  & $O(n)$ & $O(n)$ & $O((p-k)n^2)$ & $O(n)$ & $O(n)$ & $O(n)$ \\ \hline \hline
$\cU^c$ & $O(n)$ & $O(n^2)$ & poly. sol. & $O(n)$ & $O(n^2)$ & poly. sol. \\ 
	&        & $O(n \log n) $ & compact MIP & & $O(n\log n)$ & compact MIP \\ \hline
$\cU^d$ & $O(n)$ & $O(n^3)$ & compact MIP & $O(n)$ & $O(n^2)$ & compact MIP\\	
\end{tabular}
\end{small}
\end{table}

\paragraph{New results.}
All new results  for scenario sets $\cU^c$ and $\cU^d$, obtained in this paper, are summarized in Table~\ref{tab1}.  In particular, we show that all the considered problems are polynomially solvable under scenario set $\cU^c$. The polynomial algorithms for \textsc{RREC} and \textsc{R2ST} under $\cU^c$ are based on solving a polynomial number of  linear programming subproblems. We also provide polynomial time combinatorial algorithms for \textsc{AREC} and \textsc{A2ST} under both $\cU^c$ and $\cU^d$. The complexity of \textsc{RREC} and \textsc{R2ST} under $\cU^d$ remains open. For these problems we  construct compact MIP formulations and propose approximation algorithms.

\section{Continuous Budgeted Uncertainty}

In this section we address  the \textsc{RREC} and \textsc{R2ST} problems under scenario set $\cU^c$. We will show that both problems can be solved in polynomial time.

\label{sec3}

\subsection{Recoverable Robust Selection}
\label{sec31}

\subsubsection{The incremental problem}

Given $\pmb{x}\in \Phi$ and $\pmb{c}\in \cU$, the incremental problem, \textsc{IREC}, can be formulated as the following linear program (notice that the constraints $y_i\in \{0,1\}$ can be relaxed):
\begin{equation}
\label{increc}
\begin{array}{lllll}
opt_1=\min\ & \displaystyle \sum_{i\in[n]} c_iy_i \\
\text{s.t. } & \displaystyle \sum_{i\in[n]} y_i = p \\
& \displaystyle \sum_{i\in[n]} x_i y_i \ge p-k \\
& y_i\in[0,1] & i\in [n]
\end{array}
\end{equation}
It is easy to see that the \textsc{IREC} problem can be solved  in $O(n)$ time. 
Indeed,
we first choose $p-k$ items of the smallest cost  from $X_{\pmb{x}}$ and then $k$ items of the smallest cost from the remaining items. We will now show some additional properties of~(\ref{increc}), which will be used extensively later. The dual to~(\ref{increc}) is
\begin{equation}
\label{incdual}
\begin{array}{lllll}
\max\ & \displaystyle p\alpha + (p-k)\beta - \sum_{i\in[n]} \gamma_i \\
\text{s.t. } & \displaystyle \alpha + x_i \beta \le  \gamma_i + c_i &  i\in[n] \\
& \beta \ge 0 \\
& \gamma_i \ge 0 & i\in [n]
\end{array}
\end{equation}

From now on, we will assume that $k>0$ (the case $k=0$ is trivial, since $\pmb{y}=\pmb{x}$ holds).
Let $b(\pmb{c})$ be the $p$th smallest item cost for the items in $[n]$ under $\pmb{c}$ (i.e. if $c_{\sigma(1)}\leq \dots \leq c_{\sigma(n)}$ is the ordered sequence of the item costs under $\pmb{c}$, then $b(\pmb{c})=c_{\sigma(p)}$). Similarly, let $b_1(\pmb{c})$ be the $(p-k)$th smallest item cost for the items in $X_{\pmb{x}}$ and $b_2(\pmb{c})$ be the $k$th smallest item cost for the items in $\overline{X}_{\pmb{x}}$ under $\pmb{c}$.
 The following proposition characterizes the optimal values of $\alpha$ and $\beta$ in~(\ref{incdual}), and is fundamental in the following analysis:

\begin{proposition}
\label{propkk}
	Given scenario $\pmb{c}\in \cU$, the following conditions hold:
	\begin{enumerate}
		\item if $b_1(\pmb{c})\leq b(\pmb{c})$, then $\alpha=b(\pmb{c})$ and $\beta=0$ are optimal in~(\ref{incdual}),
% \todo[author=Pawel, size=\footnotesize]{ optimal in or  optimal to?} 		
		\item if $b_1(\pmb{c})>b(\pmb{c})$, then $\alpha=b_2(\pmb{c})$ and $\beta=b_1(\pmb{c})-b_2(\pmb{c})$ are optimal in~(\ref{incdual}).
	\end{enumerate}
\end{proposition}
\begin{proof}
By replacing $\gamma_i$ by $[\alpha+\beta x_i -c_i]_+$, the dual problem~(\ref{incdual}) can be represented as follows:
	\begin{equation}
\label{e1duala}
		 \max_{\alpha, \beta \geq 0} f(\alpha,\beta)= \max_{\alpha, \beta\geq 0} \left\{p\alpha+(p-k)\beta -\sum_{i\in [n]} [\alpha+\beta x_i-c_i]_+\right\},
\end{equation}
	where $[a]_+=\max\{0,a\}$.
	Let us sort the items in $[n]$ so that that $c_{\sigma(1)}\leq \dots \leq c_{\sigma(n)}$. Let us sort the items in $X_{\pmb{x}}$ so that $c_{\nu(1)}\leq \dots \leq c_{\nu(p)}$ and the items in $\overline{X}_{\pmb{x}}$ so that   $c_{\varsigma(1)}\leq \dots \leq c_{\varsigma(n-p)}$. We distinguish two cases. The first one:
	 $c_{\nu(p-k)}\leq c_{\sigma(p)}$ ($b_1(\pmb{c})\leq b(\pmb{c})$). Then it is possible to construct an optimal solution to~(\ref{increc}) with the cost equal to $\sum_{i\in [p]} c_{\sigma(i)}$. Namely, we choose $p-k$ items of the smallest costs from $X_{\pmb{x}}$ and $k$ items of the smallest cost from the remaining items. Fix $\alpha=c_{\sigma(p)}$ and $\beta=0$, which gives the case 1. By using~(\ref{e1duala}), we obtain $f(\alpha,\beta)=\sum_{i\in [p]} c_{\sigma(i)}=opt_1$ and the proposition follows from the weak duality theorem.  The second
	 case: $c_{\nu(p-k)}>c_{\sigma(p)}$ ($b_1(\pmb{c})>b(\pmb{c})$). The optimal solution to~(\ref{increc}) is then formed by the items $\nu(1),\dots,\nu(p-k)$ and $\varsigma(1),\dots,\varsigma(k)$.  Fix $\alpha=c_{\varsigma(k)}$ and $\beta=c_{\nu(p-k)}-c_{\varsigma(k)}$, which gives the case~2. By~(\ref{e1duala}), we have
	\begin{align*}
	f(\alpha,\beta) &= p\alpha +(p-k)\beta -\sum_{i\in X_{\pmb{x}}}[\alpha+\beta -c_i]_+-\sum_{i\in \overline{X}_{\pmb{x}}}[\alpha-c_i]_+ \\
	&= p c_{\varsigma(k)} +(p-k) (c_{\nu(p-k)}-c_{\varsigma(k)})-\sum_{i\in X_{\pmb{x}}} [c_{\nu(p-k)}-c_i]_+-\sum_{i\in \overline{X}_{\pmb{x}}}[c_{\varsigma(k)}-c_i]_+\\
	&= p c_{\varsigma(k)} +(p-k) (c_{\nu(p-k)}-c_{\varsigma(k)})-(p-k)c_{\nu(p-k)}+\sum_{i\in [p-k]} c_{\nu(i)}-kc_{\varsigma(k)}+\sum_{i\in [k]} c_{\varsigma(i)}\\
	&= \sum_{i\in [p-k]} c_{\nu(i)}+\sum_{i\in [k]} c_{\varsigma(i)}=opt_1 
	\end{align*}
and the proposition follows  from the weak duality theorem.
	\end{proof}

\subsubsection{The adversarial problem}
\label{secARECc}

Consider the adversarial problem~\textsc{AREC} for a given solution $\pmb{x}\in \Phi$. We will again assume that $k>0$. If $k=0$,  then all the budget $\Gamma$ is allocated to the items in $X_{\pmb{x}}$.
Scenario $\pmb{c}\in \cU^c$ which maximizes the objective value in this problem is called a \emph{worst scenario} for $\pmb{x}$ (worst scenario for short). We now give a characterization of a worst scenario.

\begin{proposition}
\label{prop2}
	There is a worst scenario $\pmb{c}=(\cl_i+\delta_i)_{i\in [n]}\in \cU^c$ such that
	\begin{enumerate}
		\item $b_1(\pmb{c})\leq b(\pmb{c})$ or
		\item $b_1(\pmb{c})$ or $b_2(\pmb{c})$ belongs to $\mathcal{D}=\{\cl_1,\dots,\cl_n,\cu_1,\dots,\cu_n\}$.
	\end{enumerate}
\end{proposition}
\begin{proof}
\begin{figure}[ht]
\centering
\includegraphics[height=5cm]{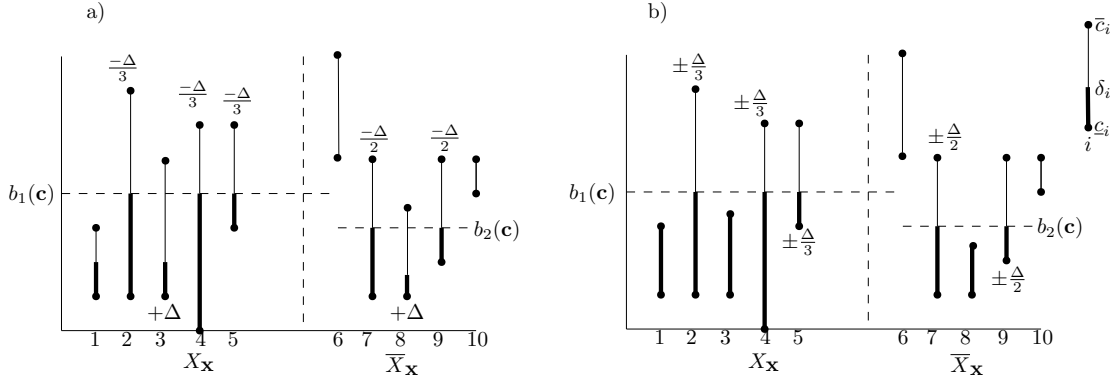}
\caption{Illustration of the proof for $n=10$,  $p=5$, $k=2$, and $X_{\pmb{x}}=\{1,\dots,5\}$.}\label{fig1}
\end{figure}
Assume that  $b_1(\pmb{c})>b(\pmb{c})$ and both $b_1(\pmb{c})$ and $b_2(\pmb{c})$ do not belong to $\mathcal{D}$.  
The main  idea of the proof is to show that there is a worst scenario
satisfying  condition~1 or~2.
Note that $b_1(\pmb{c})>b(\pmb{c})$ implies $b_1(\pmb{c})>b_2(\pmb{c})$.		  
Let $A=\{i\in X_{\pmb{x}}: \cl_i+\delta_i=b_1(\pmb{c})\}$ and $B=\{i\in \overline{X}_{\pmb{x}}: \cl_i+\delta_i=b_2(\pmb{c})\}$. Observe that $A,B\neq \emptyset$ by the definition of $b_1(\pmb{c})$ and $b_2(\pmb{c})$. Also, a positive budget must be allocated to each item in $A$ and $B$. In Figures~\ref{fig1}a and~\ref{fig1}b we have $A=\{2,4,5\}$ and $B=\{7, 9\}$.  Let $k_1$ be the number of items in $X_{\pmb{x}}$ such that $\cl_i+\delta_i<b_1(\pmb{c})$ and $k_2$ be the number of items in $\overline{X}_{\pmb{x}}$ such that $\cl_i+\delta_i<b_2(\pmb{c})$. In the sample problem (see Figures~\ref{fig1}a and~\ref{fig1}b) we have $k_1=2$ and $k_2=1$. Suppose that there is an item $j$ such that $\cl_j+\cd_j>b_1(\pmb{c})$ and $\cl_j+\delta_j<b_1(\pmb{c})$ (see the item~3 in Figure~\ref{fig1}a).  Let us transform scenario $\pmb{c}\in \cU^c$ into scenario $\pmb{c}_1\in \cU^c$ as follows: $\delta_j:=\delta_j+\Delta$ and $\delta_i:=\delta_i-\Delta/|A|$ for each $i\in A$, where $\Delta>0$ is a sufficiently small number (see Figure~\ref{fig1}a). Let $\pmb{y}$ be an optimal solution under $\pmb{c}$ and let $\pmb{y}_1$ be an optimal solution under $\pmb{c}_1$. The following equality holds
$$\pmb{c}_1\pmb{y}_1=\pmb{c}\pmb{y}+\Delta-(p-k-k_1)\frac{\Delta}{|A|}.$$
Since $|A|+k_1\geq p-k$, $\pmb{c}_1\pmb{y}_1\geq \pmb{c}\pmb{y}$ and $\pmb{c}_1$ is also a worst scenario. We can increase $\Delta$ until $b_1(\pmb{c}_1)\in \mathcal{D}$, or $b_1(\pmb{c}_1)=b_2(\pmb{c}_1)$ (which implies $b_1(\pmb{c}_1)=b(\pmb{c}_1)$), or $b_1(\pmb{c}_1)=\cl_j+\delta_j$. In the first two cases the proposition
follows and the third case will be analyzed later. The same reasoning can be applied to every item $j\in \overline{X}_{\pmb{x}}$ such that $\cl_j+\cd_j>b_2(\pmb{c})$ and $\cl_j+\delta_j<b_2(\pmb{c})$ (see the item~8 in Figure~\ref{fig1}a). So, it remains to analyze the case shown in Figure~\ref{fig1}b.  Let us again choose some sufficiently small $\Delta>0$. Define scenario $\pmb{c}_1$ by modifying $\pmb{c}$ in the following way $\delta_i:=\delta_i+\Delta/|A|$ for each $i\in A$ and $\delta_i:=\delta_i-\Delta/|B|$ for each  $i\in B$. Similarly, define scenario $\pmb{c}_2$ by modifying $\pmb{c}$ as follows $\delta_i:=\delta_i-\Delta/|A|$ for each $i\in A$ and $\delta_i:=\delta_i+\Delta/|B|$ for each  $i\in B$. Let $\pmb{y}_1$ be an optimal solution under $\pmb{c}_1$ and $\pmb{y}_2$ be an optimal solution under $\pmb{c}_2$. The following equalities hold
	 $$\pmb{c}_1\pmb{y}_1=\pmb{c}\pmb{y}+(p-k-k_1)\frac{\Delta}{|A|}-(k-k_2)\frac{\Delta}{|B|},$$
	 $$\pmb{c}_2\pmb{y}_2=\pmb{c}\pmb{y}-(p-k-k_1)\frac{\Delta}{|A|}+(k-k_2)\frac{\Delta}{|B|}.$$
Hence, either $\pmb{c}_1\pmb{y}_1\geq \pmb{c}\pmb{y}$ or $\pmb{c}_2\pmb{y}_2\geq \pmb{c}\pmb{y}$, so $\pmb{c}_1$ or $\pmb{c}_2$ is also a worst scenario.  
We can now increase $\Delta$ until $\pmb{c}_1$ ($\pmb{c}_2$) satisfies condition 1 or 2.
\end{proof}

Using~(\ref{incdual}) and the definition of scenario set $\cU^c$, we can represent \textsc{AREC} as the following linear programming problem:
\begin{equation}
\label{adv0}
\begin{array}{llll}
\max\ & p\alpha + (p-k)\beta - \sum_{i\in[n]} \gamma_i \\
\text{s.t. } & \alpha + x_i \beta \le \gamma_i+ \cl_i + \delta_i & \forall i\in[n] \\
& \sum_{i\in[n]} \delta_i \le \Gamma \\
& \delta_i \le d_i & \forall i\in[n] \\
&  \beta \ge 0 \\
& \gamma_i, \delta_i \ge 0 & i\in[n]
\end{array}
\end{equation}
Thus \textsc{AREC} can be solved in polynomial time. In the following we will construct
 a~strongly polynomial combinatorial algorithm for solving \textsc{AREC}. The following corollary is a consequence of Proposition~\ref{propkk} and Proposition~\ref{prop2}:

\begin{corollary}
\label{cor1}
	There is an optimal solution to~(\ref{adv0}) in which
	\begin{enumerate}
		\item $\beta=0$ or
		\item $\alpha$ or $\alpha+\beta$ belongs to $\mathcal{D}=\{\cl_1,\dots,\cl_n,\cu_1,\dots,\cu_n\}$.
	\end{enumerate}
\end{corollary}
\begin{proof}
According to Proposition~\ref{propkk}, there is an optimal solution to~(\ref{adv0}) which induces a worst scenario $\pmb{c}=(\cl_i+\delta_i)_{i\in [n]}\in \cU^c$, which satisfies conditions~1 and~2 of Proposition~\ref{propkk}.  If the condition~1 is fulfilled, i.e. $b_1(\pmb{c})\leq b(\pmb{c})$, then according to Proposition~\ref{prop2} we get $\beta=0$. If $b_1(\pmb{c})>b(\pmb{c})$, then condition~2 from Proposition~\ref{propkk} and condition~2 from Proposition~\ref{prop2}  hold. Both these conditions imply the condition~2 of the corollary.
\end{proof}

\begin{proposition}
\label{propbase}
	The optimal values of $\alpha$ and $\beta$ in~(\ref{adv0}) can be found by solving the following problem:
\begin{equation}
\label{mipaa2}
	\max_{\alpha, \beta\geq 0}\left\{ \alpha p + \beta (p-k)-\max\left \{\sum_{i\in [n]} [\alpha+\beta x_i-\cl_i]_+-\Gamma, \sum_{i \in [n]} [\alpha+\beta x_i -\cu_i]_+\right\}\right\}
\end{equation}
\end{proposition}
\begin{proof}
	Let us first rewrite~(\ref{adv0}) in the following way:
\begin{equation}
	\label{mipaa1}
	\begin{array}{lllll}
		\max &\displaystyle p\alpha+(p-k)\beta -\sum_{i\in [n]} [\alpha+\beta x_i-\cl_i-\delta_i]_+\\
			\text{s.t. } & \displaystyle \sum_{i\in [n]} \delta_i\leq \Gamma & \\
			& 0\leq \delta_i\leq \cd_i & i\in [n]  \\
			&\beta\geq 0 \\
	\end{array}
\end{equation}
	Let us fix $\alpha$ and $\beta\geq 0$ in~(\ref{mipaa1}). Then the optimal values of $\delta_i$ can be then found by solving the following subproblem:
	\begin{equation}
	\label{subp0}
	\begin{array}{lllll}
		z=\min &\displaystyle \sum_{i\in [n]} [\alpha+\beta x_i-\cl_i-\delta_i]_+ \\
			\text{s.t. } & \displaystyle \sum_{i\in [n]} \delta_i\leq \Gamma \\
			& 0\leq \delta_i\leq \cd_i & i\in [n] \\
	\end{array}
\end{equation}
	Let $U=\sum_{i\in [n]} [\alpha+\beta x_i-\cl_i]_+$. Observe that $[U-\Gamma]_+$ is a lower bound on $z$ as $z\geq 0$ and it is not possible to decrease $U$ by more than $\Gamma$.
	The subproblem~(\ref{subp0}) can be solved by applying the following greedy method. For  $i:=1,\dots, n$, if $\alpha+\beta x_i-\cl_i>0$, we fix $\delta_i=\min\{\alpha+\beta x_i-\cl_i, \cd_i,\Gamma\}$ and modify $\Gamma:=\Gamma-\delta_i$. If, at some step, $\Gamma=0$ we have reached the  lower bound. Hence $z=[U-\Gamma]_+$. On the other hand if, after the algorithm terminates, we still have $\Gamma>0$, then $z=\sum_{i\in [n]}[\alpha +\beta x_i-\cl_i-\cd_i]_+$. In consequence
	\begin{align*}
	z &= \max\left\{[U-\Gamma]_+, \sum_{i\in [n]}[\alpha +\beta x_i-\cl_i-\cd_i]_+\right\} \\
	&= \max\left \{\sum_{i\in [n]} [\alpha+\beta x_i-\cl_i]_+-\Gamma, \sum_{i \in [n]} [\alpha+\beta x_i -\cu_i]_+\right\},
	\end{align*}
	which together with~(\ref{mipaa1}) completes the proof.
\end{proof}
Having the optimal values of $\alpha$ and $\beta$, the worst scenario $\pmb{c}=(\cl_i+\delta_i)_{i\in [n]}$, can be found in $O(n)$ time by  applying the greedy method to~(\ref{subp0}), described in the proof of Proposition~\ref{propbase}. We now construct an efficient algorithm for solving~(\ref{mipaa2}), which will give us the optimal values of $\alpha$ and $\beta$.  We will illustrate this algorithm by using the sample problem shown in Figure~\ref{fig2}.
 \begin{figure}[ht]
	\centering
	\includegraphics{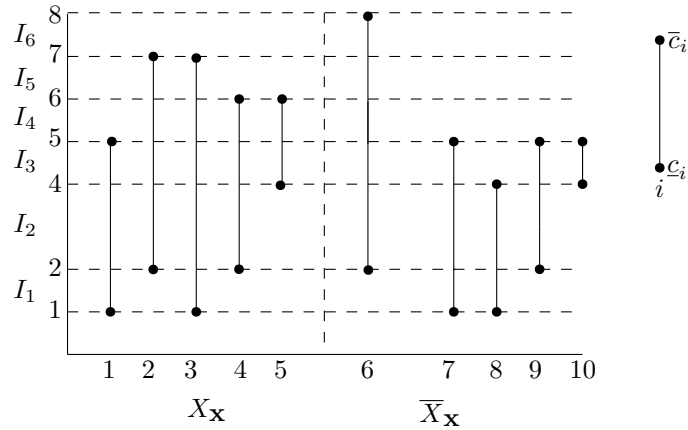}
	\caption{A sample problem with $n=10$, $p=5$, $k=2$, $\Gamma=24$, and $X_{\pmb{x}}=\{1,\dots,5\}$.}\label{fig2}
	\end{figure}

Let $h_{(1)}\leq h_{(2)}\leq \dots \leq h_{(l)}$ be the ordered sequence of the distinct values from $\mathcal{D}$. This sequence defines a family of  closed intervals $I_j=[h_{(j)}, h_{(j+1)}]$, $j\in [l-1]$, which partitions the interval $[\min_{i\in[n]}  \cl_i ,\max_{i\in[n]} \cu_i]$. Notice that $l\leq 2n$. In the example shown in Figure~\ref{fig2} we have six intervals $I_1,\dots I_6$ which split the interval $[1,8]$. 

By Corollary~\ref{cor1}, we need to investigate two cases.
In the first case, we have $\beta=0$. Then~(\ref{mipaa2}) reduces to the following problem:
\begin{equation}
\label{adva}
\max_{\alpha}f(\alpha)=\max_{\alpha} \left\{ \alpha p -\max\left \{\sum_{i\in [n]} [\alpha-\cl_i]_+-\Gamma, \sum_{i \in [n]} [\alpha-\cl_i -\cd_i]_+\right\}\right\}
\end{equation}
Consider the problem of maximizing $f(\alpha)$ over a fixed interval $I_j$. It is easy to verify that~(\ref{adva}) reduces then to finding the maximum of  a minimum of two linear functions of $\alpha$ over $I_j$. For example, when $\alpha \in I_3=[4,5]$, then after an easy computation, the problem~(\ref{adva}) becomes
$$\max_{\alpha \in [4,5]}\min\{-5\alpha+44,4\alpha+4\}.$$
It is well known that the maximum  value of $\alpha$ is attained at one of the bounds of the interval $I_j$ or at the intersection point of the two linear functions of $\alpha$. In this case we compute $\alpha$ by solving $-5\alpha+44=\alpha+4$ which yields $\alpha=4.44$.
We can now solve~(\ref{adva}) by solving at most $2n$ subproblems consisting in maximizing $f(\alpha)$ over $I_1,\dots I_l$. Notice, however, that in some cases we do not have to examine all the intervals $I_1,\dots, I_l$. We can use the fact that $\alpha$ is the $p$th smallest item cost in the computed scenario. In the example, the optimal value of $\alpha$ belongs to $I_1\cup I_2\cup I_3$. The function $f(\alpha)$ for the sample problem is shown in Figure~\ref{fig2}. The optimal value of $\alpha$ is 4.44. The scenario corresponding to $\alpha=4.44$ can be obtained by applying a greedy method and it is also shown in Figure~\ref{fig2x}.

\begin{figure}[ht]
	\centering
	\includegraphics{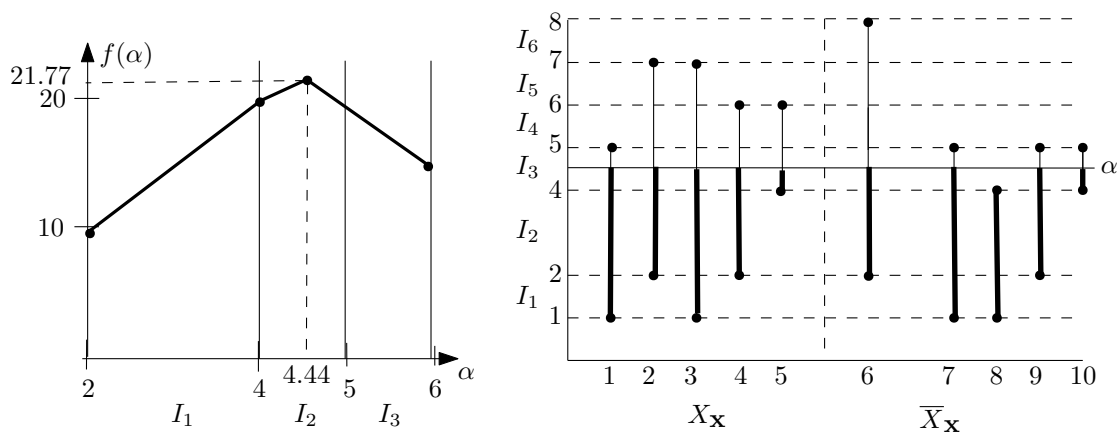}
	\caption{The function $f(\alpha)$ for the sample problem and the worst scenario for the optimal value of $\alpha=4.44$.}\label{fig2x}
	\end{figure}

We now discuss the second case in Corollary~\ref{cor1}. Let us fix $\gamma=\alpha+\beta$ and rewrite~(\ref{mipaa2}) as follows:
\begin{equation}
\label{mipaa3}
\begin{array}{lll}
	\displaystyle \max_{\alpha, \gamma\geq\alpha} g(\alpha, \gamma)=\max_{\alpha, \gamma\geq \alpha}
	\{  \alpha k + \gamma (p-k)-\\
	\displaystyle \left.\max \left\{\sum_{i\in X_{\pmb{x}}} [\gamma-\cl_i]_++ \sum_{i\in \overline{X}_{\pmb{x}}}[\alpha-\cl_i]_+-\Gamma, \sum_{i \in X_{\pmb{x}}} [\gamma -\cl_i -\cd_i]_++\sum_{i\in \overline{X}_{\pmb{x}}} [\alpha-\cl_i-\cd_i]_+\right\} \right\}.
\end{array}
\end{equation}
 According to Corollary~\ref{prop2}, the optimal value of $\alpha$ or $\gamma$ belongs to $\mathcal{D}$. So, let us first fix $\gamma\in \mathcal{D}$ and consider the problem $\max_{\alpha} g(\alpha,\gamma)$.  The optimal value of $\alpha$ can be found by optimizing $\alpha$ over each interval $I_j$, whose upper bound is not greater then $\gamma$ (it follows from the constraint $\alpha\leq \gamma$). Again, the problem $\max_{\alpha\in I_j} g(\alpha, \gamma)$ can be reduced to 
 maximizing a minimum of two linear functions of $\alpha$ over a closed interval. To see this consider the sample problem shown in Figure~\ref{fig2}. Fix $\gamma=6$ and assume that $\alpha\in I_2$. Then, 
a trivial verification shows that 
$$\max_{\alpha\in [2,4]} g(\alpha,6)=\max_{\alpha\in [2,4]}\min\{28-2\alpha,2\alpha+17\}.$$
The maximum is attained when $28-2\alpha=2\alpha+17$, so for $\alpha=2.75$. The function $g(\alpha,6)$ is shown in Figure~\ref{fig2y}. It attains the maximum in the interval $I_2$ at $\alpha=2.75$. The scenario which corresponds to $\alpha=2.75$ and $\gamma=6$ is also shown in Figure~\ref{fig2y}. In the same way we can find the optimal value of $\alpha$ for each fixed $\gamma\in\mathcal{D}$. Since $\gamma$ is the $(p-k)$th smallest item cost  in $X_{\pmb{x}}$ under the computed scenario, not all values of $\gamma$ in $\mathcal{D}$ need to be examined. In the example we have to only try $\gamma\in\{2,4,5,6\}$. 
\begin{figure}[ht]
	\centering
	\includegraphics{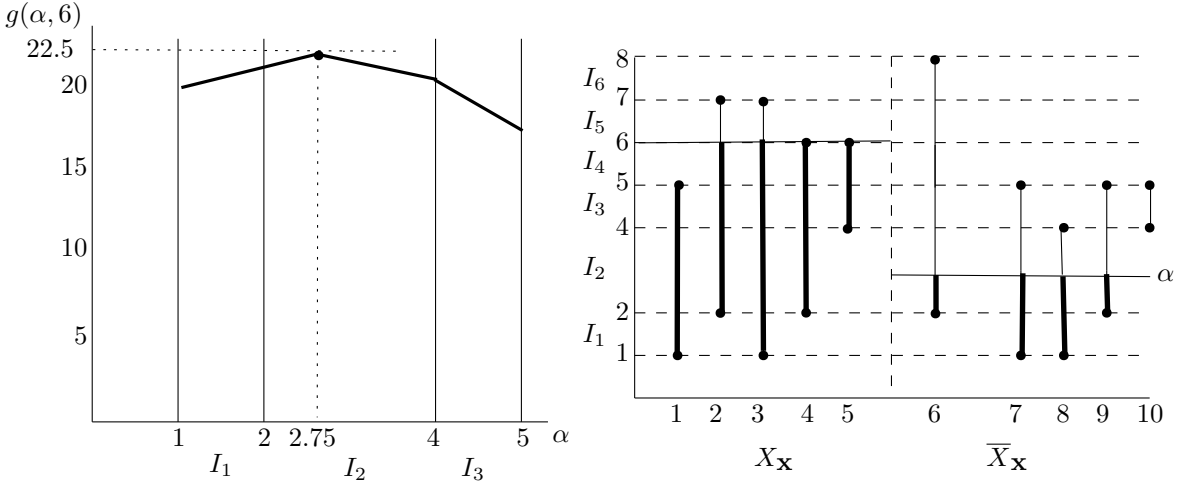}
	\caption{The function $g(\alpha,6)$ for the sample problem and the worst scenario for the optimal value of $\alpha=2.75$.}\label{fig2y}
	\end{figure}

We can then repeat the reasoning for every fixed $\alpha\in\mathcal{D}$. Namely, we solve the problem $\max_{\gamma\geq \alpha} g(\alpha,\gamma)$ by solving the problem for each interval whose lower bound is not less than $\alpha$. Again, not all values of $\alpha\in \mathcal{D}$ need to be examined. Since $\alpha$ is the $k$th smallest item cost in $\overline{X}_{\pmb{x}}$, we should check only the values of $\alpha\in\{1,2,4,5\}$.

\begin{theorem}
	The problem \textsc{AREC} under scenario set $\cU^c$ can be solved in $O(n^2)$ time.
\end{theorem}
\begin{proof}
	We will present a sketch of the $O(n^2)$ algorithm. We first determine the family of intervals $I_1,\dots,I_l$, which requires $O(n\log n)$ time. Now, the key observation is that we can evaluate first all the sums that appear in~(\ref{adva}) and~(\ref{mipaa3}) for each interval $I_j$. 
We can compute and store these sums for every $I_j$ in $O(n^2)$ time. Now each problem $\max_{\alpha\in I_j} g(\alpha,\gamma)$ for  $\gamma \in \mathcal{D}$, $\max_{\gamma \in I_j}g(\alpha,\gamma)$ for $\alpha \in \mathcal{D}$, and $\max_{\alpha \in I_j} f(\alpha)$  can be solved in constant time by inserting the computed earlier sums into~(\ref{mipaa3}) and~(\ref{adva}). The number of problems that must be solved is $O(n^2)$, so the overall running time of the algorithm is $O(n^2)$.
\end{proof}

Using a more refined analysis and data structures such as min-heaps (see, e.g., \cite{CO90}), this $O(n^2)$ result can be further improved to $O(n\log n)$.
We present the proof in Appendix~\ref{sec:appendix}.
% The proof can be found in the technical report of this paper (see \cite{chassein2017recoverable}).

We now show the following proposition, which will be used later:
\begin{proposition}[\textbf{Dominance rule}]
\label{propdom}
	Let $k,l$ be two items such that $\cl_k\leq \cl_l$ and $\cu_k \leq \cu_l$. Let $x_l=1$ and $x_k=0$ in~(\ref{adv0}). Then the maximum objective value in~(\ref{adv0}) will not increase when we change $x_l=0$ and $x_k=1$.
\end{proposition}
\begin{proof}
Let $X_{\pmb{x}'}=X_{\pmb{x}}\cup\{k\}\setminus \{l\}$. Notice that~(\ref{adva}) does not depend on the first stage solution $\pmb{x}$, so it remains to investigate the effect of replacing $X_{\pmb{x}}$ with $X_{\pmb{x}'}$ in~(\ref{mipaa3}).
It is enough to show that for each $\alpha$ and $\gamma\geq \alpha$ the following inequalities hold:
\begin{equation}
\label{f1}
U_1=[\gamma-\cl_k]_+-[\gamma-\cl_l]_++[\alpha-\cl_l]_+-[\alpha-\cl_k]_+\geq 0
\end{equation}
and
\begin{equation}
U_2=[\gamma-\cu_k]_+-[\gamma-\cu_l]_++[\alpha-\cu_l]_+-[\alpha-\cu_k]_+\geq 0
\end{equation}
Inequality~(\ref{f1}) can be proven by distinguishing the following cases:
	if $\alpha \leq \gamma\leq \cl_k \leq \cl_l$, then $U_1=0$;
	 if $\alpha\leq \cl_k\leq \gamma \leq \cl_l$, then $U_1=\gamma-\cl_k\geq 0$;
	 if $\alpha\leq \cl_k \leq \cl_l\leq \gamma$, then  $U_1=\cl_l-\cl_k\geq 0$;
	 if $\cl_k\leq \alpha\leq \gamma\leq \cl_l$,  then $U_1=\gamma-\alpha\geq 0$;
	 if $\cl_k\leq \alpha \leq \cl_l \leq \gamma$, then $U_1=\gamma-\cl_k-\gamma+\cl_l-\alpha+\cl_k=\cl_l-\alpha\geq 0$;
	 if $\cl_k\leq \cl_l\leq \alpha \leq \gamma$, then $U_1=0$.
The proof of the fact that $U_2\geq 0$ is just the same.
\end{proof}

\subsubsection{The recoverable robust problem}
\label{secRRCalg}

In this section we study \textsc{RREC} under scenario set $\cU^c$. We first identify some special cases of this problem, which are known to be polynomially solvable. If $\Gamma$ is sufficiently large, say $\Gamma\geq \sum_{i\in [n]} \cd_i$, then scenario set $\cU^c$ reduces to $\cU^I$ and the problem can be solved in $O((p-k)n^2)$ time~\cite{KZ15b}. Also the boundary cases $k=0$ and $k=p$ are polynomially solvable. When $k=p$, then we choose in the first stage $p$ items of the smallest costs under~$\pmb{C}$. The total cost of this solution can be then computed in $O(n^2)$ time by solving the corresponding adversarial problem. If $k=0$, then \textsc{RREC} is equivalent to the \textsc{MinMax} problem with cost intervals $[C_i+\cl_i, C_i+\cu_i]$, $i\in [n]$. Hence it is polynomially solvable due to the results obtained in~\cite{BS03}.

Consider now the general case with any $k\in [p]$. We first show a method of preprocessing a given instance of the problem.  Given two items $i,j\in [n]$, we write $i\preceq j$ if $C_i \leq C_j$, $ \cl_i\leq \cl_j$ and  $\cu_i\leq \cu_j$.
For any fixed item $l\in [n]$, suppose that $|\{k: k\preceq l\}|\geq p$. Let an optimal solution $\pmb{x}\in \Phi$ to \textsc{RREC} be given, in which $x_l=1$. There is an item $x_k$ such that $k\preceq l$ and $x_k=0$ in $\pmb{x}$. We form solution $\pmb{x}'$ by setting $x_k=1$ and $x_l=0$. From Proposition~\ref{propdom} and inequality $C_k\leq C_l$, we get
$$\pmb{C}\pmb{x}'+\max_{\pmb{c}\in \cU^c} \min_{\pmb{y}\in \Phi^k_{\pmb{x}'}} \pmb{c}\pmb{y}\leq \pmb{C}\pmb{x}+\max_{\pmb{c}\in \cU^c} \min_{\pmb{y}\in \Phi^k_{\pmb{x}}} \pmb{c}\pmb{y},$$
and $\pmb{x}$ is also an optimal solution to \textsc{RREC}. In what follows, we can remove $l$ from $[n]$ without violating the optimum obtaining (after renumbering the items) a smaller item set $[n-1]$. We can now repeat iteratively this reasoning, which allows us to reduce the size of the input instance.  Also, for each $l\in [n]$, if $|\{k:l\preceq k\}|\geq n-p$, then we do not violate the optimum after setting $x_l=1$.

We now reconsider the adversarial problem~(\ref{adv0}). Its dual is the following:
\begin{align*}
\min\ & \sum_{i\in[n]} \cl_i y_i + \Gamma \pi + \sum_{i\in[n]} \cd_i \rho_i \\
\text{s.t. } & \sum_{i\in[n]} y_i = p \\
& \sum_{i\in[n]} x_iy_i \ge p-k \\
& \pi + \rho_i \ge y_i  & i\in [n]\\
& y_i \in [0,1] & i\in [n]\\
& \pi \ge 0 \\
& \rho_i \ge 0 & i\in[n]
\end{align*}
Using this formulation, we can represent the \textsc{RREC} problem under scenario set $\cU^c$
as the following compact mixed-integer program:
\begin{equation}
\label{mipreccont}
\begin{array}{lllll}
\min\ & \displaystyle \sum_{i\in[n]} C_i x_i + \sum_{i\in[n]} \cl_i y_i + \Gamma \pi + \sum_{i\in[n]} \cd_i \rho_i \\
\text{s.t. } & \displaystyle \sum_{i\in[n]} y_i = p \\
& \displaystyle\sum_{i\in[n]} x_i = p \\
& \displaystyle \sum_{i\in[n]} x_iy_i \ge p-k \\
& \pi + \rho_i \ge y_i & i\in [n] \\
& x_i \in\{0,1\} & i\in [n] \\
& y_i \in [0,1] & i\in [n] \\
& \pi \ge 0 \\
& \rho_i \ge 0 & i\in[n]
\end{array}
\end{equation}
The products $x_i y_i$, $i\in [n]$, can be linearized by using standard methods, which leads to a linear MIP formulation for \textsc{RREC}. Before solving this model the preprocessing described earlier can be applied.
We now show that~(\ref{mipreccont}) can be solved in polynomial time.
Notice that we can assume $\pi\in[0,1]$ and $\rho_i = [y_i - \pi]_+$. Let us split each variable $y_i=\yl_i+\yu_i$, where $\yl_i\in [0,\pi]$ is the cheaper, and $\yu_i\in[0,1-\pi]$ is the more expensive part of $y_i$ (through the additional costs of $\rho_i$). The resulting formulation is then
\begin{equation}
\label{submip1}
\begin{array}{lllll}
\min\ &\displaystyle \sum_{i\in[n]} C_i x_i + \sum_{i\in[n]} \cl_i \yl_i + \sum_{i\in[n]} \cu_i \yu_i + \Gamma\pi \\
\text{s.t. } & \displaystyle\sum_{i\in[n]} (\yl_i + \yu_i) = p \\
& \displaystyle\sum_{i\in[n]} x_i = p \\
&\displaystyle \sum_{i\in[n]} x_i(\yl_i+\yu_i) \ge p-k \\
& x_i \in\{0,1\} & i\in[n]\\
& \yl_i \in [0,\pi] & i\in[n]\\
& \yu_i \in [0,1-\pi] & i\in[n]\\
& \pi \in [0,1]
\end{array}
\end{equation}
Observe that if $\yu_i>0$, then $\yl_i=\pi$ for each $i\in [n]$ in some optimal solution, as the whole cheaper part of each item is taken first.
Substituting $\ysl_i \pi$ into $\yl_i $ and  $\ysu_i (1-\pi)$ into $\yu_i$, we can write equivalently
\begin{equation}
\label{submip2}
\begin{array}{lllll}
\min\ & \displaystyle\sum_{i\in[n]} C_i x_i + \sum_{i\in[n]} \pi \cl_i \ysl_i + \sum_{i\in[n]} (1-\pi) \cu_i \ysu_i + \Gamma\pi \\
\text{s.t. } & \displaystyle\sum_{i\in[n]} (\pi \ysl_i + (1-\pi) \ysu_i) = p \\
& \displaystyle\sum_{i\in[n]} x_i = p \\
& \displaystyle\sum_{i\in[n]} x_i( \pi \ysl_i+ (1-\pi) \ysu_i) \ge p-k \\
& x_i \in\{0,1\} & i\in [n]\\
& \ysl_i, \ysu_i\in [0,1] & i\in [n]\\
& \pi \in [0,1]
\end{array}
\end{equation}
Again, in some optimal solution, if $\ysu_i>0$, then $\ysl_i=1$  (and if $\ysl_i<1$, then $\ysu_i=0$) for each $i\in [n]$. The following lemma characterizes the optimal solution to~(\ref{submip2}):
\begin{lemma}\label{lem8}
There is an optimal solution to~\eqref{submip2} which satisfies the following properties:
\begin{enumerate}
	\item at most one variable in $\{\ysl_1,\dots,\ysl_n, \ysu_1, \dots, \ysu_n\}$ is fractional,
	\item $\pi=\frac{q}{r}$ for $r\in [n+1]$, $q\in[n]\cup\{0\}$.
\end{enumerate}
\end{lemma}

\begin{proof}
Let $\pmb{x}^*\in \Phi$ be optimal in~\eqref{submip2}. Fix $\pmb{x}=\pmb{x}^*$ in~\eqref{submip1} and consider the problem with additional slack variables: 
\begin{align*}
\min\ &\sum_{i\in[n]} \cl_i \yl_i + \sum_{i\in[n]} \cu_i \yu_i + \Gamma\pi\\
\text{s.t. } & \sum_{i\in[n]} (\yl_i + \yu_i) = p \\
& \sum_{i\in X_{\pmb{x}^*}} (\yl_i+\yu_i) - \delta = p-k \\
& \yl_i + \alpha_i = \pi & i\in[n]\\
& \yu_i + \beta_i = 1-\pi & i\in[n] \\
& \pi + \gamma = 1 \\
& \yl_i, \yu_i, \alpha_i, \beta_i  \ge 0 & i\in[n] \\
& \delta,\pi,\gamma \ge 0
\end{align*}
This problem contains $2n+3$ constraints and $4n+3$ variables. Thus, there is an optimal solution with $2n+3$ basis variables and $2n$ non-basis variables. The following cases are possible.

\begin{itemize}
\item First let us assume $\pi=0$. Then, $\yl_i = 0$ for all $i\in[n]$ and for the resulting problem
\begin{align*}
\min\ & \sum_{i\in[n]} \cu_i \yu_i \\
\text{s.t. } & \sum_{i\in[n]} \yu_i = p \\
& \sum_{i\in X_{\pmb{x}^*}} \yu_i \ge p-k \\
& \yu_i \in[0,1] & i\in[n]
\end{align*}
there exists an optimal integer solution $\pmb{\yu}$ (by first taking the $p-k$ cheapest items from $X_{\pmb{x}^*}$, and completing the solution with the $k$ cheapest items from $[n]$). Since $\pi=0$, we get $\pmb{z}=\pmb{1}$ and $\pmb{\overline{z}}=\pmb{\yu}$ and the claim is shown. The proof for $\pi=1$ is analogous.
\item Assume that $0 < \pi < 1$, so both $\pi$ and $\gamma$ are basis variables. If $\yl_i$ is a non-basis variable, then $\alpha_i$ is a basis-variable (and vice versa). The same holds for $\yu_i$ and $\beta_i$. The following cases result:
\begin{enumerate}
\item If $\delta$ is a basis variable, then it follows that the $2n$ non-basis variables are found in $\pmb{\yl}$, $\pmb{\yu}$, $\pmb{\alpha}$ and $\pmb{\beta}$. Hence, $\yl_i \in\{0,\pi\}$ and $\yu_i \in\{0,1-\pi\}$, so $z_i,\overline{z}_i\in\{0,1\}$ for each $i\in [n]$. Let $\cY=\sum_{i\in [n]} \ysl_i$ and $\cZ=\sum_{i\in [n]} \ysu_i$. 
We then get $\pi = (p-\cZ)/(\cY-\cZ)$ (see model~(\ref{submip2}))
 and point~2  of the lemma is proven (note that if $\cY=\cZ$, we can assume $\pi=1$).
\item Let us assume that $\delta$ is a non-basis variable. Then one of two cases must hold:
\begin{enumerate}
\item There is $j$ such that both $\yl_j$ and $\alpha_j$ are basis variables. Then, all other $\yl_i$ are either $0$ or $\pi$ ($\ysl_i\in \{0,1\}$), and all other $\yu_i$ are either $0$ or $1-\pi$ ($\ysu_i\in \{0,1\}$). In terms of formulation~\eqref{submip2}, we have $\ysl_j \in[0,1]$ and only $z_j$ can be fractional (i.e. other than 0 or 1). 
In order to show the second point of the lemma, let us define $\cY = \sum_{i\in[n]\setminus\{j\}} \ysl_i$, $\cZ= \sum_{i\in[n]} \ysu_i$, $\cY' = \sum_{i\in X_{\pmb{x}^*}\setminus\{j\}} \ysl_i$, and $\cZ' = \sum_{i\in X_{\pmb{x}^*}} \ysu_i$. Since $\ysl_i\geq \ysu_i$ for each $i\in [n]$ and $\ysu_j=0$, 
the inequalities $\cY\geq \cZ$ and $\cY'\geq \cZ'$ hold.
We consider now the subproblem of~(\ref{submip2}) that reoptimizes the solution only in $\pi$ and $\ysl_j$:
\begin{align*}
\min\ & \cl_j \pi \ysl_j + t \pi \\
\text{s.t. } & \pi \ysl_j = p - \cZ - (\cY-\cZ)\pi \\
& \pi \cY' + (1-\pi)\cZ' + \pi\ysl_j \ge p-k \\
& \ysl_j \in [0,1] \\
& \pi\in [0,1],
\end{align*}
where $t=\left( \Gamma + \sum_{i\in[n]\setminus\{j\}} \cl_i \ysl_i - \sum_{i\in[n]} \cu_i\ysu_i \right)$ is a constant.
We remove variable $\ysl_j$ using the equality $\ysl_j = (p - \cZ - (\cY-\cZ)\pi) / \pi$ from the problem. The constraint $\ysl_j \ge 0$ becomes $\pi \le (p-\cZ)/(\cY-\cZ)$, while the constraint $\ysl_j \le 1$ is $\pi \ge (p-\cZ)/(\cY-\cZ+1)$. Hence, the reoptimization problem becomes
\begin{align*}
\min\ & t \pi + \cl_j (p-\cZ-(\cY-\cZ)\pi)\\
\text{s.t. } & \pi((\cY'-\cZ')-(\cY-\cZ)) \ge \cZ-\cZ'-k  \\
& \frac{p - \cZ}{1+ \cY-\cZ} \le \pi \le \frac{p - \cZ}{ \cY-\cZ} \\
& \pi\in [0,1]
\end{align*}
We can conclude that the optimal value of $\pi$ is one of $\frac{p-\cZ}{\cY-\cZ}$, $\frac{p-\cZ}{\cY-\cZ+1}$, or $\frac{\cZ-\cZ'-k}{(\cY'-\cZ')-(\cY-\cZ)}$. Since $\cZ,\cZ',\cY, \cY'$ are all integers  from $0$ to $n$, the second point of the lemma is true.

\item There is $j$ such that both $\yu_j$ and $\beta_j$ are basis variables. This case is analogue to the previous case.
\end{enumerate}
\end{enumerate}
\end{itemize}

\end{proof}

\begin{lemma}\label{lemma8}
Problem \textsc{RREC} under $\cU^c$ can be solved by solving the problem
\begin{equation}
\label{lem9}
\begin{array}{lllll}
\min\ & \displaystyle\sum_{i\in\cJ} C_i x_i + \sum_{i\in\cJ} \pi \cl_i \ysl_i + \sum_{i\in\cJ}(1-\pi)\cu_i\ysu_i \\
\text{s.t. } & \displaystyle\sum_{i\in\cJ} x_i = p \\
&\displaystyle \sum_{i\in\cJ} \ysl_i  = \cY \\
& \displaystyle\sum_{i\in\cJ} \ysu_i  = \cZ \\
& \displaystyle\sum_{i\in\cJ} \ysl'_i \ge \cY' \\
& \displaystyle\sum_{i\in\cJ} \ysu'_i \ge \cZ' \\
& \ysl'_i \le x_i & i\in\cJ \\
& \ysl'_i \le \ysl_i & i\in\cJ \\
& \ysu'_i \le x_i & i\in\cJ \\
& \ysu'_i \le \ysu_i & i\in\cJ \\
& x_i \in \{0,1\} & i\in\cJ \\
& \ysl_i,  \ysu_i, \ysl'_i, \ysu'_i \in \{0,1\} & i\in\cJ
\end{array}
\end{equation}
for polynomially many sets $\cJ$ and values of  $\cY$, $\cZ$, $\cY'$, $\cZ'$ and $\pi$.
\end{lemma}
\begin{proof}
Using Lemma~\ref{lem8}, we first consider the case when $\ysl_i,\ysu_i\in\{0,1\}$ for all $i\in[n]$. Then we set $\cJ = [n]$ and guess the values  $\cY=\sum_{i\in\cJ} \ysl_i$ and $\cZ = \sum_{i\in \cJ} \ysu_i$. We set $\pi = (p-\cZ)/(\cY-\cZ)$, and further guess all possible values of $\cX = \sum_{i\in \cJ} x_i \ysl_i$ and $\coX = \sum_{i\in \cJ} x_i \ysu_i$ for which the constraint $\pi \cX + (1-\pi)\coX \ge p-k$ is fulfilled. In total, we have to try polynomially many values. For each resulting problem we linearize $\ysl_i x_i$ and $\ysu_i x_i$ and we get~\eqref{lem9}.

Assume now that  some $\ysl_j \in[0,1]$ can be fractional (notice that in this case  we can fix $\ysu_j=0$).  We guess the index $j$, the value of $\pi$, and the value of $x_j$. We fix then  $\cJ = [n]\setminus\{j\}$ and continue as in the first part of the proof. Namely we guess $\cY$, $\cZ$,  and $\cX$, $\coX$ for the fixed $\pi$, and construct the problem~\eqref{lem9}.
Notice that the value of $z_j$ can be retrieved from $\pi z_j=(p - \cZ - (\cY-\cZ) \pi)$.
The case where $\ysu_j \in[0,1]$ can be fractional is analogue. Again, we have to try polynomially many values.
To solve problem~\eqref{submip2}, we then take the best of all solutions.
\end{proof}

\begin{lemma}\label{lemma9}
For fixed $\cJ$, $\cY$, $\cZ$, $\cY'$, $\cZ'$ and $\pi$, the problem~\eqref{lem9} can be solved in polynomial time.
\end{lemma}
\begin{proof}
We will show that the coefficient matrix of the relaxation of \eqref{lem9} is totally unimodular. We will use the following Ghouila-Houri criterion~\cite{C63}: An $m\times n$ integral matrix is totally unimodular, if and only if for each set of rows $R=\{r_1,\ldots,r_K\}\subseteq[m]$ there exists a coloring (called a valid coloring) $l(r_i) \in\{-1,1\}$ such that the weighted sum of every column restricted to $R$ is $-1$, $0$, or $1$. 
For simplicity, we assume w.l.o.g. that $\cJ=[n]$. Note that it is enough to show that the coefficient matrix of~(\ref{lem9}) without the relaxed constraints $x_i, \ysl_i,  \ysu_i, \ysl'_i, \ysu'_i\leq 1$ is totally unimodular. The matrix, together with a labeling of its rows, is shown in Table~\ref{coefflem9}.
\begin{table}[ht]
\caption{The coefficient matrix of~\eqref{lem9}}\label{coefflem9}
\resizebox{\textwidth}{!}{
\begin{tabular}{c| | c c c c  | c c c c  | c c c c  | c c c c  | c c c c }
 & $x_1$ & $x_2$ &   $\cdots$ & $x_n$ & $\ysl_1$ & $\ysl_2$ &  $\cdots$ & $\ysl_n$ & $\ysu_1$ & $\ysu_2$ &   $\cdots$ & $\ysu_n$ & $\ysl'_1$ & $\ysl'_2$ & $\cdots$ & $\ysl'_n$ & $\ysu'_1$ & $\ysu'_2$  & $\cdots$ & $\ysu'_n$ \\
 
\hline
$a_1$ & 1 & 1 &  $\cdots$ & 1 & 0 & 0 &  $\cdots$ & 0 & 0 & 0 &  $\cdots$ & 0 & 0 & 0 &  $\cdots$ & 0 & 0 & 0 & $\cdots$ & 0 \\
$a_2$ & 0 & 0 &  $\cdots$ & 0 & 1 & 1 &  $\cdots$ & 1 & 0 & 0 &  $\cdots$ & 0 & 0 & 0 &  $\cdots$ & 0 & 0 & 0 &  $\cdots$ & 0 \\
$a_3$ & 0 & 0 &  $\cdots$ & 0 & 0 & 0 &  $\cdots$ & 0 & 1 & 1 &  $\cdots$ & 1 & 0 & 0 & $\cdots$ & 0 & 0 & 0 &  $\cdots$ & 0 \\
$a_4$ & 0 & 0 &  $\cdots$ & 0 & 0 & 0 &  $\cdots$ & 0 & 0 & 0 &  $\cdots$ & 0 & 1 & 1 &  $\cdots$ & 1 & 0 & 0 &  $\cdots$ & 0 \\
$a_5$ & 0 & 0 &  $\cdots$ & 0 & 0 & 0 &  $\cdots$ & 0 & 0 & 0 &  $\cdots$ & 0 & 0 & 0 &  $\cdots$ & 0 & 1 & 1 & $\cdots$ & 1 \\
\hline

$b_1$ & 1 & 0 &  $\cdots$ & 0 & 0 & 0 &  $\cdots$ & 0 & 0 & 0 &  $\cdots$ & 0 & -1 & 0 &  $\cdots$ & 0 & 0 & 0 &  $\cdots$ & 0 \\
$b_2$ & 0 & 1 &  $\cdots$ & 0 & 0 & 0 &  $\cdots$ & 0 & 0 & 0 &  $\cdots$ & 0 & 0 & -1 &  $\cdots$ & 0 & 0 & 0 &  $\cdots$ & 0 \\
$\vdots$ & $\vdots$ & $\vdots$  && $\vdots$&  $\vdots$ & $\vdots$ & & $\vdots$ &
 $\vdots$ & $\vdots$ & & $\vdots$& $\vdots$ & $\vdots$ & & $\vdots$&
  $\vdots$ & $\vdots$ & & $\vdots$\\
$b_n$ & 0 & 0 &  $\cdots$ & 1 & 0 & 0 &  $\cdots$ & 0 & 0 & 0 &  $\cdots$ & 0 & 0 & 0 &  $\cdots$ & -1 & 0 & 0 &  $\cdots$ & 0 \\
\hline

$c_1$ & 0 & 0 &  $\cdots$ & 0 & 1 & 0 &  $\cdots$ & 0 & 0 & 0 &  $\cdots$ & 0 & -1 & 0 &  $\cdots$ & 0 & 0 & 0 &  $\cdots$ & 0 \\
$c_2$ & 0 & 0 &  $\cdots$ & 0 & 0 & 1 &  $\cdots$ & 0 & 0 & 0 &  $\cdots$ & 0 & 0 & -1 &  $\cdots$ & 0 & 0 & 0 &  $\cdots$ & 0 \\
$\vdots$ & $\vdots$ & $\vdots$  && $\vdots$&  $\vdots$ & $\vdots$ & & $\vdots$ &
 $\vdots$ & $\vdots$ & & $\vdots$& $\vdots$ & $\vdots$ & & $\vdots$&
  $\vdots$ & $\vdots$ & & $\vdots$\\
$c_n$ & 0 & 0 &  $\cdots$ & 0 & 0 & 0 &  $\cdots$ & 1 & 0 & 0 &  $\cdots$ & 0 & 0 & 0 &  $\cdots$ & -1 & 0 & 0 &  $\cdots$ & 0 \\
\hline

$d_1$ & 1 & 0 &  $\cdots$ & 0 & 0 & 0 &  $\cdots$ & 0 & 0 & 0 &  $\cdots$ & 0 & 0 & 0 &  $\cdots$ & 0 & -1 & 0 &  $\cdots$ & 0 \\
$d_2$ & 0 & 1 &  $\cdots$ & 0 & 0 & 0 &  $\cdots$ & 0 & 0 & 0 &  $\cdots$ & 0 & 0 & 0 &  $\cdots$ & 0 & 0 & -1 &  $\cdots$ & 0 \\
$\vdots$ & $\vdots$ & $\vdots$  && $\vdots$&  $\vdots$ & $\vdots$ & & $\vdots$ &
 $\vdots$ & $\vdots$ & & $\vdots$& $\vdots$ & $\vdots$ & & $\vdots$&
  $\vdots$ & $\vdots$ & & $\vdots$\\
$d_n$ & 0 & 0 &  $\cdots$ & 1 & 0 & 0 &  $\cdots$ & 0 & 0 & 0 &  $\cdots$ & 0 & 0 & 0 &  $\cdots$ & 0 & 0 & 0 &  $\cdots$ & -1 \\
\hline

$e_1$ & 0 & 0 &  $\cdots$ & 0 & 0 & 0 &  $\cdots$ & 0 & 1 & 0 &  $\cdots$ & 0 & 0 & 0 &  $\cdots$ & 0 & -1 & 0 &  $\cdots$ & 0 \\
$e_2$ & 0 & 0 &  $\cdots$ & 0 & 0 & 0 &  $\cdots$ & 0 & 0 & 1 &  $\cdots$ & 0 & 0 & 0 & $\cdots$ & 0 & 0 & -1 &  $\cdots$ & 0 \\
$\vdots$ & $\vdots$ & $\vdots$  && $\vdots$&  $\vdots$ & $\vdots$ & & $\vdots$ &
 $\vdots$ & $\vdots$ & & $\vdots$& $\vdots$ & $\vdots$ & & $\vdots$&
  $\vdots$ & $\vdots$ & & $\vdots$\\
$e_n$ & 0 & 0 & $\cdots$ & 0 & 0 & 0 &  $\cdots$ & 0 & 0 & 0 &  $\cdots$ & 1 & 0 & 0 &  $\cdots$ & 0 & 0 & 0 &  $\cdots$ & -1\

\end{tabular}}
\end{table}

Given a set of rows $R$ we use the following algorithm to color the rows in $R$:

\begin{enumerate}

\item $l(d_i)=1$ for each $d_i\in R$

\item If $a_5 \in R$, set $l(a_5) = 1$, $l(e_i) = 1$ for each $e_i\in R$ and $l(a_3) = -1$.

\item If $a_5 \notin R$, set $l(e_i) = -1$ for each $e_i\in R$ and $l(a_3) = 1$.

\item If $a_1\in R$, set $l(a_1) = -1$ and $l(b_i) = 1$ for each $b_i\in R$. 

\begin{enumerate}
\item If $a_4\in R$, set $l(a_4) = 1$, $l(c_i) = 1$ for each $c_i\in R$ and $l(a_2) = -1$.

\item If $a_4 \notin R$, set $l(c_i) = -1$ for each $c_i\in R$ and $l(a_2) = 1$.

\end{enumerate}

\item If $a_1 \notin R$, set $l(b_i) = -1$ for each $b_i\in R$.

\begin{enumerate}

\item If $a_4\in R$, set $l(a_4) = -1$, $l(c_i) = -1$ for each $c_i\in R$ and $l(a_2) = 1$.

\item If $a_4 \notin R$, set $l(c_i) = 1$ for each $c_i\in R$ and $l(a_2) = -1$.

\end{enumerate}

\end{enumerate}

If $a_1 \in R$, then $l(a_1)=-1$ and the rows $b_i, d_i$, $\in R$ have always color 1; if $a_1\notin R$, then the rows $b_i\in R$ have color -1 and the rows $d_i\in R$ have color 1. So the coloring is valid for all columns corresponding to $x_i$. In order to prove that the coloring is valid for the columns corresponding to $\ysl_1, \dots, \ysl_n$,  $\ysu_1,  \dots,  \ysu_n$ it is enough to observe that the algorithm always assigns different colors to $a_2$ and the rows $c_i\in R$, and $a_3$ and the rows in $e_i\in R$. If $a_4\in R$, then $a_4$ has the same color as all $b_i\in R$ or all $c_i\in R$; if $a_4\notin R$, then $b_i\in R$ and $c_i\in R$ have different color. In consequence, the coloring is valid for the columns corresponding to $\ysl'_1, \dots, \ysl'_n$. It is also easy to see that the coloring is valid for the columns corresponding to variables $\ysu'_1, \dots,\ysu'_n$ (see steps 1-3).

\end{proof}

\begin{theorem}
\label{thmreccont}
The RREC problem under scenario set $\cU^c$ is solvable in polynomial time.
\end{theorem}
\begin{proof}
This result is a direct consequence of Lemma~\ref{lemma8} and Lemma~\ref{lemma9}.
\end{proof}

\subsection{Two-Stage Robust Selection}
\label{sec32}

In this section we investigate the two-stage model, namely the \textsc{I2ST}, \textsc{A2ST} and \textsc{R2ST} problems under scenario set $\cU^c$. In order to solve \textsc{A2ST} we will use the results obtained for \textsc{AREC}. We will also show that \textsc{R2ST} is polynomially solvable as it can be reduced to solving a polynomial number of linear programming problems.

\subsubsection{The incremental and adversarial problems}
\label{sec32contia}

We are given a first stage solution $\pmb{x}\in \Phi_1$ with $|X_{\pmb{x}}|=p_1$, where $p_1\in [p]\cup\{0\}$. Define $\tilde{p}=p-p_1$. The incremental problem, \textsc{I2ST}, can be solved in $O(n)$ time. It is enough to choose $\tilde{p}$ items of the smallest costs out of $\overline{X}_{\pmb{x}}$ under the given scenario $\pmb{c}$.  On the other hand, the adversarial problem, \textsc{A2ST}, can be reduced to \textsc{AREC}.  We first remove from $[n]$ all the items belonging $X_{\pmb{x}}$, obtaining (after an appropriate renumbering) the item set $[n-p_1]$. We then fix $k=\tilde{p}$. As we can exchange all the items, the choice of the first stage solution in \textsc{AREC} is irrelevant. Consider the formula~(\ref{mipaa3}) for the constructed instance of \textsc{AREC}. The optimal value of $\gamma$ satisfies $\gamma=\alpha$ and~(\ref{mipaa3}) becomes:
\begin{equation}
\label{adv2stc}
\max_{\alpha} \left\{ \alpha \tilde{p} -\max\left \{\sum_{i\in [\tilde{p}]} [\alpha-\cl_i]_+-\Gamma, \sum_{i \in [\tilde{p}]} [\alpha-\cu_i]_+\right\}\right\},
\end{equation}
which is in turn the same as~(\ref{adva}). Problem~(\ref{adv2stc}) can be solved in $O(n^2)$ time, by the method described in Section~\ref{secARECc}. This means that \textsc{A2ST} is solvable in $O(n^2)$ time.

\subsubsection{The two-stage robust problem}

Given $\pmb{x}\in \Phi_1$ and $\pmb{c}\in \cU$, the incremental problem, \textsc{I2ST}, can be formulated as the following linear program:
\begin{equation}
\label{inc2stc}
\begin{array}{lllll}
\min\ & \displaystyle \sum_{i\in[n]} c_i y_i \\
\text{s.t.} & \displaystyle \sum_{i\in[n]} (y_i + x_i )= p \\
& y_i \le 1-x_i & i\in[n] \\
& y_i \in \{0,1\} & i\in[n]
\end{array}
\end{equation}
We can now relax the integrality constraints in~(\ref{inc2stc}) and dualize this problem, to find a compact formulation for the adversarial problem, \textsc{A2ST}, under scenario set $\cU^c$:
\begin{equation}
\label{adv2st}
\begin{array}{lllll}
\max \ & \displaystyle (p-\sum_{i\in[n]} x_i) \alpha - \sum_{i\in[n]} (1-x_i) \gamma_i \\
\text{s.t. } & \alpha \le \delta_i + \gamma_i + \cl_i & i\in[n] \\
& \delta_i \le \cd_i & i\in[n] \\
& \displaystyle \sum_{i\in[n]} \delta_i \le \Gamma \\
& \gamma_i,  \delta_i \ge 0 & i\in[n] 
\end{array}
\end{equation}
Dualizing~\eqref{adv2st}, we get the following problem:
\begin{align*}
\min \ & \sum_{i\in[n]} \cl_i y_i + \Gamma \pi + \sum_{i\in[n]} \cd_i\rho_i \\
\text{s.t. } & \sum_{i\in[n]} (y_i + x_i) = p \\
& \pi + \rho_i \ge y_i & i\in[n] \\
& y_i \le 1 - x_i & i\in [n]\\
& \pi \ge 0 \\
& \rho_i \ge 0 & i\in [n]\\
& y_i \in [0,1] & i\in[n]
\end{align*}
which can be used to construct the following mixed-integer program for \textsc{R2ST}:
\begin{equation}
\label{rec2stmip}
\begin{array}{llll}
\min \ & \displaystyle \sum_{i\in[n]} C_i x_i + \sum_{i\in[n]} \cl_i y_i + \Gamma \pi + \sum_{i\in[n]} \cd_i\rho_i \\
\text{s.t. } & \displaystyle \sum_{i\in[n]} (y_i + x_i) = p \\
& x_i + y_i \le 1 &  i\in[n] \\
& \pi + \rho_i \ge y_i &  i\in[n] \\
& \pi \ge 0 \\
&\rho_i \ge 0 & i\in[n] \\
& x_i \in \{0,1\} & i\in[n]\\
& y_i \in [0,1] & i\in[n]
\end{array}
\end{equation}
%This can problem can also be considered as the robust counterpart of a mixed-integer program affected by budgeted uncertainty on the continuous variables. The complexity of such problems is open in general, see \cite{A06a}.
We now show that~(\ref{rec2stmip}) can be solved in polynomial time.  We first apply to~(\ref{rec2stmip}) similar transformation as for the \textsc{RREC} problem (see Section~\ref{secRRCalg}),
%Note that $\pi\in[0,1]$. We set $\rho_i := [y_i - \pi]_+$ to remove the variables $\pmb{\rho}$. Then, we split $\pmb{y}$ variables into the (low-cost) part $\pmb{\yl}$ from $0$ up to $\pi$ and the (high-cost) part $\pmb{\yu}$ from $1-\pi$ to $1$, 
which results in the following equivalent formulation:
\begin{equation}
\label{mip2st1}
\begin{array}{lllll}
\min\ & \displaystyle\sum_{i\in[n]} C_i x_i + \sum_{i\in[n]} \cl_i \yl_i + \sum_{i\in[n]} \cu \yu_i + \Gamma \pi \\
\text{s.t. } & \displaystyle\sum_{i\in[n]} (x_i + \yl_i + \yu_i)= p \\
& x_i + \yl_i \le 1 &  i\in[n] \\
& x_i + \yu_i \le 1 &  i\in[n] \\
& x_i\in\{0,1\} & i\in[n]\\
& \yl_i \in [0,\pi] & i\in[n]\\
& \yu_i \in [0,1-\pi] & i\in[n]\\
& \pi \in [0,1]
\end{array}
\end{equation}
Again, by setting $\ysl_i\pi = \yl_i$ and $\ysu_i (1-\pi) = \yu_i$, we rescale the variables and find the following equivalent, nonlinear problem:
\begin{equation}
\label{mip2st2}
\begin{array}{lllll}
\min\ & \displaystyle \sum_{i\in[n]} C_i x_i + \sum_{i\in[n]} \pi \cl_i \ysl_i + \sum_{i\in[n]} (1-\pi) \cu_i \ysu_i + \Gamma \pi \\
\text{s.t. } &  \displaystyle\sum_{i\in[n]} (x_i + \pi \ysl_i + (1-\pi) \ysu_i) = p \\
& x_i + \ysl_i \le 1 &  i\in[n] \\
& x_i + \ysu_i \le 1 &  i\in[n] \\
& x_i\in\{0,1\} &  i\in[n]\\
& \ysl_i, \ysu_i \in [0,1] &  i\in[n]\\
& \pi \in [0,1]
\end{array}
\end{equation}
Note that we can write $x_i + \ysl_i \le 1$ instead of $x_i + \pi \ysl_i \le 1$ and $x_i+\ysu_i \le 1$ instead of $x_i+ (1-\pi) \ysu_i \leq 1$ for each $i\in [n]$.

\begin{lemma}\label{lem1}
There exists an optimal solution to~\eqref{mip2st2} in which
\begin{enumerate}
 \item $\ysl_i, \ysu_i \in \{0,1\}$ for all $i\in[n]$,
 \item $\pi=\frac{q}{r}$, where $q\in [p] \cup \{0\}$, $r\in [n]$.
 \end{enumerate}
\end{lemma}

\begin{proof}
We first prove point 1.
Let $\pmb{x}^*\in \Phi_1$ be optimal in~(\ref{mip2st2}). Using formulation~\eqref{mip2st1}, we consider the following linear program for fixed $\pmb{x}=\pmb{x}^*$ and additional slack variables:
\begin{align*}
\min\ &  \sum_{i\in\overline{X}_{\pmb{x}^*}} \cl_i \yl_i + \sum_{i\in\overline{X}_{\pmb{x}^*}} \cu_i \yu_i + \Gamma\pi\\
\text{s.t. } & \sum_{i\in\overline{X}^*} (\yl_i + \yu_i )= p - \sum_{i\in[n]} x^*_i\\
& \yl_i + \alpha_i = \pi  &  i\in\overline{X}_{\pmb{x}^*}\\
& \yu_i + \beta_i = 1- \pi  &  i\in\overline{X}_{\pmb{x}^*}\\
& \pi + \gamma = 1 \\
& \yl_i, \yu_i, \alpha_i, \beta_i\ge 0 &  i\in\overline{X}_{\pmb{x}^*}\\
&\pi,\gamma \ge 0
\end{align*}
This problem has $4|\overline{X}_{\pmb{x}*}|+2$ variables and $2|\overline{X}_{\pmb{x}^*}|+2$ constraints. Thus, there is an optimal solution with $2|\overline{X}_{\pmb{x}^*}|+2$ basis variables and $2|\overline{X}_{\pmb{x}^*}|$ non-basis variables. 
If $\pi=0$, then $\yl_i=0$ for all $i\in\overline{X}_{\pmb{x}^*}$ and the problem becomes a selection problem in $\pmb{\yu}$, for which there is an optimal integer solution. If $\pi =1$, then $\yu_i =0$ for each $i\in [n]$ and the problem becomes a selection problem in $\pmb{\yl}$. Hence, in both these cases, there exists an optimal solution to~\eqref{mip2st2} that is integer in $\pmb{\ysl}$ and $\pmb{\ysu}$.
 So let us assume  that $\pi >0$ and $\pi < 1$, i.e., both $\pi$ and $\gamma$ are basis variables. Note that whenever $\yl_i$ (resp. $\yu_i$) is a non-basis variable, then $\alpha_i$ (resp. $\beta_i$) is a basis variable, and vice versa. Hence, all variables $\yl_i$ are either $0$ or $\pi$, and all variables $\yu_i$ are either $0$ or $1-\pi$. This corresponds to a solution where all $\ysl_i$ and $\ysu_i$ are either $0$ or $1$ in formulation~\eqref{mip2st2}.

We now prove point 2. Let  $x^*_i, \ysl^*_i,\ysu^*_i\in\{0,1\}$, $i\in [n]$,  be  optimal in~(\ref{mip2st2}). If $\sum_{i\in[n]} \ysl^*_i = \sum_{i\in[n]} \ysu^*_i$, then there exists an optimal solution with $\pi^*\in\{0,1\}$. So let us assume $\sum_{i\in[n]} \ysl^*_i > \sum_{i\in[n]} \ysu^*_i$ (recall that $\ysl^*_i\geq \ysu^*_i$ for each $i\in [n]$). By rearranging terms, we obtain
\[ \pi =\frac{p - \sum_{i\in[n]} (x^*_i + \ysu^*_i)}{\sum_{i\in[n]} (\ysl^*_i - \ysu^*_i )}. \]
Write $\cX= \sum_{i\in[n]} x^*_i$, $\cY= \sum_{i\in[n]} \ysl^*_i$ and $\cZ = \sum_{i\in[n]} \ysu^*_i$. We have $\cX\in\{0,\ldots,p\}$, $\cY,\cZ\in\{0,\ldots,n\}$. Note that if for an item $i$ we have $\ysu^*_i = 1$, then also $\ysl^*_i = 1$. Consequently, $\pi$ is of the form described in point~2 of the lemma.
\end{proof}

\begin{lemma}\label{lem3}
The \textsc{R2ST} problem under $\cU^c$ can be solved by solving problem
\begin{equation}
\label{mipl4}
\begin{array}{lllll}
\min\ & \displaystyle \sum_{i\in [n]} C_i x_i + \sum_{i\in [n]} \pi \cl_i \ysl_i + \sum_{i\in [n]} (1-\pi)\cu_i \ysu_i \\
\text{s.t. } & \displaystyle \sum_{i\in [n]} x_i = \cX \\
& \displaystyle \sum_{i\in [n]} \ysl_i = \cY \\
& \displaystyle \sum_{i\in [n]} \ysu_i = \cZ \\
& x_i + \ysl_i \le 1 & i\in[n]\\
& x_i + \ysu_i \le 1 & i\in[n]\\
& x_i,\ysl_i,\ysu_i \in \{0,1\} & i\in[n]
\end{array}
\end{equation}
for polynomially many values of $\cX, \cY, \cZ$ and $\pi$.
\end{lemma}

\begin{proof}
Using Lemma~\ref{lem1}, we will try all possible values of $\pi$ and for each fixed $\pi$ we will find an optimal solution to~(\ref{mip2st2}) where all $\pmb{x}$, $\pmb{\ysl}$ and $\pmb{\ysu}$ are integer. Let the value $\pi=\pi^*$ be fixed. The resulting problem is then
\begin{align*}
\min\ & \sum_{i\in[n]} C_i x_i + \sum_{i\in[n]} \pi^* \cl_i \ysl_i + \sum_{i\in[n]} (1-\pi^*) \cu_i \ysu_i + \Gamma \pi^* \\
\text{s.t. } & \sum_{i\in[n]} x_i + \pi^* \sum_{i\in[n]} \ysl_i + (1-\pi^*) \sum_{i\in[n]}  \ysu_i = p \\
& x_i + \ysl_i \le 1 &  i\in[n] \\
& x_i + \ysu_i \le 1 &  i\in[n] \\
& x_i,\ysl_i,\ysu_i \in \{0,1\} & i\in[n] 
\end{align*}
As $\pi^*$ is fixed, we can enumerate all possible values of $\cX=\sum_{i\in[n]} x_i$, $\cY=\sum_{i\in[n]} \ysl_i$ and $\cZ=\sum_{i\in[n]} \ysu_i$ that generate this value $\pi^*$, i.e., we enumerate all possible solutions to $\cX+ \pi^* \cY+ (1-\pi^*) \cZ = p$. There can be at most $p$ choices of $\cX$ and at most $n$ choices of $\cY$ and $\cZ$. This leads to the problem~\eqref{mipl4}.
%Additionally, we set $\pmb{c}^1 = \pmb{C}$, $\pmb{c}^2 = \pi^*\pmb{c}$ and $\pmb{c}^3 = (1-\pi^*)(\pmb{c}+\pmb{d})$ to arrive at the required form~\eqref{mipl4}. 
By choosing the best of the computed solutions, we then find an optimal solution to \textsc{R2ST}.
\end{proof}

\begin{lemma}\label{lem4}
For fixed $\cX, \cY, \cZ$ and $\pi$,  the problem~\eqref{mipl4} can be solved in polynomial time.
\end{lemma}
\begin{proof}
We prove that the coefficient matrix of the relaxation of~\eqref{mipl4} is totally unimodular. We will use the Ghouila-Houri criterion~\cite{C63} (see the proof of Lemma~\ref{lemma9}). The coefficient matrix of the constraints of~\eqref{mipl4} is shown in Table~\ref{l5table} (we can skip the relaxed constraints $x_i, \ysl_i, \ysu_i\leq 1$).
\begin{table}[htb]
\centering
\caption{coefficient matrix of problem~\eqref{mipl4}.}\label{l5table}
\begin{tabular}{c | c c c c c | c c c c c | c c c c c }
 & $x_1$ & $x_2$ & $x_3$ & $\cdots$ & $x_n$ & $\ysl_1$ & $\ysl_2$ & $\ysl_3$ & $\cdots$ & $\ysl_n$ & $\ysu_1$ & $\ysu_2$ & $\ysu_3$ & $\cdots$ & $\ysu_n$  \\
 \hline
$a_1$ & 1 & 1 & 1 & $\cdots$ & 1 & 0 & 0 & 0 & $\cdots$ & 0 & 0 & 0 & 0 & $\cdots$ & 0 \\
$a_2$ & 0 & 0 & 0 & $\cdots$ & 0 & 1 & 1 & 1 & $\cdots$ & 1 & 0 & 0 & 0 & $\cdots$ & 0 \\
$a_3$ & 0 & 0 & 0 & $\cdots$ & 0 & 0 & 0 & 0 & $\cdots$ & 0 & 1 & 1 & 1 & $\cdots$ & 1 \\
\hline
$b_1$ & 1 & 0 & 0 & $\cdots$ & 0 & 1 & 0 & 0 & $\cdots$ & 0 & 0 & 0 & 0 & $\cdots$ & 0 \\
$b_2$ & 0 & 1 & 0 & $\cdots$ & 0 & 0 & 1 & 0 & $\cdots$ & 0 & 0 & 0 & 0 & $\cdots$ & 0 \\
$b_3$ & 0 & 0 & 1 & $\cdots$ & 0 & 0 & 0 & 1 & $\cdots$ & 0 & 0 & 0 & 0 & $\cdots$ & 0 \\
$\vdots$ & $\vdots$ & $\vdots$ & $\vdots$ & & $\vdots$ & $\vdots$ & $\vdots$ & $\vdots$ & & $\vdots$ & $\vdots$ & $\vdots$ & $\vdots$ & & $\vdots$\\
$b_n$ & 0 & 0 & 0 & $\cdots$ & 1 & 0 & 0 & 0 & $\cdots$ & 1 & 0 & 0 & 0 & $\cdots$ & 0 \\
\hline
$c_1$ & 1 & 0 & 0 & $\cdots$ & 0 & 0 & 0 & 0 & $\cdots$ & 0 & 1 & 0 & 0 & $\cdots$ & 0 \\
$c_2$ & 0 & 1 & 0 & $\cdots$ & 0 & 0 & 0 & 0 & $\cdots$ & 0 & 0 & 1 & 0 & $\cdots$ & 0 \\
$c_3$ & 0 & 0 & 1 & $\cdots$ & 0 & 0 & 0 & 0 & $\cdots$ & 0 & 0 & 0 & 1 & $\cdots$ & 0 \\
$\vdots$ & $\vdots$ & $\vdots$ & $\vdots$ & & $\vdots$ & $\vdots$ & $\vdots$ & $\vdots$ & & $\vdots$ & $\vdots$ & $\vdots$ & $\vdots$ & & $\vdots$\\
$c_n$ & 0 & 0 & 0 & $\cdots$ & 1 & 0 & 0 & 0 & $\cdots$ & 0 & 0 & 0 & 0 & $\cdots$ & 1
\end{tabular}
\end{table}

Let us choose a subset of the rows $R=A\cup B \cup C$ with $A\subseteq\{a_1,a_2,a_3\}$, $B\subseteq \{b_1,\ldots,b_n\}$ and $C\subseteq\{c_1,\ldots,c_n\}$.
We now determine the coloring for $R$ in the following way:
\begin{itemize}
\item If $A=\emptyset$, then $l(b_i) = -1$, $l(c_i) = 1$.

\item If $A=\{a_1\}$, then $l(a_1) = -1$, $l(b_i)=l(c_i)=1$.

\item If $A=\{a_2\}$, then $l(a_2) = -1$, $l(b_i) = 1$, $l(c_i) = -1$.

\item If $A=\{a_1,a_2\}$, then $l(a_1)=l(a_2) = -1$, $l(b_i)=l(c_i) =1$.

\item If $A=\{a_3\}$, then $l(a_3) = -1$, $l(b_i) = -1$, $l(c_i) = 1$.

\item If $A=\{a_1,a_3\}$, then $l(a_1)=l(a_3) = -1$, $l(b_i)=l(c_i) = 1$.

\item If $A=\{a_2,a_3\}$, then $l(a_2)=-1$, $l(a_3)=1$, $l(b_i)=1$, $l(c_i)=-1$.

\item If $A=\{a_1,a_2,a_3\}$, then $l(a_i) = -1$, $l(b_i)=l(c_i) = 1$.
\end{itemize}
It is easy to verify that the coloring is valid for each of these cases.
\end{proof}

\begin{theorem}
The \textsc{R2ST} problem under scenario set $\cU^c$ is solvable in polynomial time.
\end{theorem}
\begin{proof}
 A direct consequence of Lemma~\ref{lem3} and Lemma~\ref{lem4}.
\end{proof}

\section{Discrete Budgeted Uncertainty}

\label{sec4}

In this section we consider the \textsc{RREC} and \textsc{R2ST} problems under scenario set $\cU^d$. We will use some results obtained for the continuous budgeted uncertainty (in particular Proposition~\ref{propkk}). Notice also that the incremental problems \textsc{IREC} and \textsc{I2ST} are the same as for the continuous case.

\subsection{Recoverable Robust Selection}

\subsubsection{The adversarial problem}
\label{secadvdisc}

Let us fix solution~$\pmb{x}\in\Phi$. The adversarial problem, \textsc{AREC}, under scenario set $\cU^d$, can be  represented as the following mathematical programming problem:
\begin{equation}
\label{dbcadv5}
\begin{array}{llll}
\displaystyle \max_{\pmb{\delta}\in\{0,1\}^n \atop \sum_{i\in[n]} \delta_i \le \Gamma }& \displaystyle\min_{\pmb{y}}\ &\displaystyle \sum_{i\in[n]} (\cl_i+\cd_i \delta_i)y_i  \\
&\text{s.t. } & \displaystyle \sum_{i\in[n]} y_i = p \\
&& \displaystyle \sum_{i\in[n]} x_i y_i \ge p-k \\
&& y_i \in \{0,1\} &  i\in[n]
\end{array}
\end{equation}
Relaxing the integrality constraints $y_i\in \{0,1\}$ in~(\ref{dbcadv5}) for the inner incremental problem, and
taking  the dual of it, we obtain the following integer linear program for \textsc{AREC}:

\begin{equation}
\label{arecdisc0}
\begin{array}{lll}
\max\ &\displaystyle p\alpha + (p-k) \beta - \sum_{i\in[n]} \gamma_i \\
\text{s.t. } & \alpha + x_i \beta - \gamma_i \le  \cl_i+\cd_i \delta_i  &  i\in[n]\\
&\displaystyle \sum_{i\in[n]} \delta_i \leq \Gamma \\
& \beta \ge 0 \\
& \gamma_i \ge 0 &  i\in[n] \\
& \delta_i \in \{0,1\} &  i\in[n]
\end{array}
\end{equation}
Let $\pmb{\delta}^{*}\in \{0,1\}^n$ be optimal in~(\ref{arecdisc0}). The vector $\pmb{\delta}^*$ describes the worst scenario $\hat{\pmb{c}}=(\hat{c}_i)_{i\in[n]}=(\cl_i+\cd_i \delta^{*}_i)_{i\in [n]}\in \cU^d$. When we fix this scenario in~(\ref{arecdisc0}), then we get the problem~(\ref{incdual}), discussed in Section~\ref{sec31}, to which Proposition~\ref{propkk} can be applied. Since $\hat{c}_i$ is either $\cl_i$ or $\cu_i$ for each $i\in [n]$, only a finite number of values of $\alpha$ and $\beta$ need to be considered as optimal to~(\ref{arecdisc0}) (see Proposition~\ref{propkk}).
In the following, we will show how to find these values efficiently.
We can fix $\gamma_i =  [\alpha + x_i \beta - \cd_i\delta_i -\cl_i]_+$ for each $i\in[n]$ in~(\ref{arecdisc0}).
Hence, $\gamma_i=[ \alpha + x_i \beta - \cd_i - \cl_i]_+$ if  $\delta_i = 1$, and $\gamma_i= [\alpha + x_i \beta -\cl_i]_+$ if $\delta_i=0$.
%\[ \gamma_i = \begin{cases}
%[ \alpha + x_i \beta - d_i -c_i]_+ & \text{ if } z_i = 1 \\
%[\alpha + x_i \beta -c_i]_+ & \text{ if } z_i = 0
%\end{cases} \]
In consequence,~(\ref{arecdisc0}) can be rewritten as follows:
\begin{equation}
\label{arecdisc}
\begin{array}{llll}
\max\ & \displaystyle p\alpha + (p-k) \beta - \sum_{i\in[n]}  [ \alpha + x_i \beta -\cl_i]_+ \\
& \displaystyle+\sum_{i\in[n]} (  [ \alpha + x_i \beta -\cl_i]_+ - [ \alpha + x_i \beta - \cd_i -\cl_i]_+) \delta_i  \\
\text{s.t. } & \displaystyle \sum_{i\in[n]} \delta_i \leq\Gamma  \\
& \delta_i \in \{0,1\} &  i\in[n] \\
& \beta \ge 0
\end{array}
\end{equation}

It is easily seen that for fixed $\alpha$, $\beta$,  and $\pmb{x}$, (\ref{arecdisc}) is the 
 \textsc{Selection} problem, which can be solved in $O(n)$ time.
% Thus solving problem~(\ref{arecdisc}) boils down to finding 
%$\alpha$ and $\beta$ that maximize the objective of~(\ref{arecdisc}). 
We now find the sets of relevant values of $\alpha$ and $\beta$.
Let us order the elements in $[n]$ according to their costs~$\hat{c}_i$ and the cost bounds 
$\cl_i$ and $\cu_i$ for $i\in [n]$ in the following way:
\[
 \hat{c}_{\sigma(1)}\leq\cdots\leq \hat{c}_{\sigma(n)},\mbox{ }
 \cl_{\underline{\sigma}(1)}\leq\cdots\leq \cl_{\underline{\sigma}(n)}, \mbox{ }
  \cu_{\overline{\sigma}(1)}\leq\cdots\leq \cu_{\overline{\sigma}(n)}.
\]
Similarly, let us  order the elements in~$X_{\pmb{x}}$  so that
\[
 \hat{c}_{\nu(1)}\leq\cdots\leq \hat{c}_{\nu(p)},\mbox{ }
 \cl_{\underline{\nu}(1)}\leq\cdots\leq \cl_{\underline{\nu}(p)}, \mbox{ }
  \cu_{\overline{\nu}(1)}\leq\cdots\leq \cu_{\overline{\nu}(p)}
\]
and in $\overline{X}_{\pmb{x}}$, namely
\[
 \hat{c}_{\varsigma(1)}\leq\cdots\leq \hat{c}_{\varsigma(n-p)},\mbox{ }
 \cl_{\underline{\varsigma}(1)}\leq\cdots\leq \cl_{\underline{\varsigma}(n-p)}, \mbox{ }
  \cu_{\overline{\varsigma}(1)}\leq\cdots\leq \cu_{\overline{\varsigma}(n-p)}.
\]
According to Proposition~\ref{propkk},  $\alpha=\hat{c}_{\sigma(p)}$ and $\beta=0$, or $\alpha=\hat{c}_{\varsigma(k)}$, and $\beta=\hat{c}_{\nu(p-k)}-\hat{c}_{\varsigma(k)}$ are optimal in~(\ref{arecdisc}).
%The following  proposition is the adaptation of Proposition~\ref{propkk} to problem (\ref{dbcadv1})-(\ref{dbcadv6}).
%Its proof goes in similar manner.
%\begin{proposition}
%\label{pdbcadv}
%	Given scenario an optimal solution $(\pmb{z}^{*},\pmb{y}^{*})$
%	to (\ref{dbcadv1})-(\ref{dbcadv6}) and 
%	scenario~$\hat{\pmb{c}}=(c_i+d_i z^{*}_i)_{i\in [n]}$ the following conditions hold:
%	\begin{enumerate}
%		\item if $\hat{c}_{\nu(p-k)}\leq \hat{c}_{\sigma(p)}$, then $\pmb{z}^{*}$,
%		$\alpha^{*}=\hat{c}_{\sigma(p)}$ and $\beta^{*}=0$ are optimal to~(\ref{fdbcadv1})-(\ref{fdbcadv5}),
%		\item if $\hat{c}_{\nu(p-k)}>\hat{c}_{\varsigma(k)}$
%		(equivalently $\hat{c}_{\nu(p-k)}> \hat{c}_{\sigma(p)}$), then 
%		$\pmb{z}^{*}$,
%		$\alpha^{*}=\hat{c}_{\varsigma(k)}$ and 
%		$\beta^{*}=\hat{c}_{\nu(p-k)}-\hat{c}_{\varsigma(k)}$ are optimal to~(\ref{fdbcadv1})-(\ref{fdbcadv5}).
%	\end{enumerate}
%\end{proposition}
Thus we have
\begin{align*}
 &\hat{c}_{\sigma(p)}\in \mathcal{C}_{\sigma(p)}=
 \{ \cl_{\underline{\sigma}(p)},\ldots, \cl_{\underline{\sigma}(p+\Gamma)}\}\cup
 \{  \cu_{\overline{\sigma}(1)},\ldots,   \cu_{\overline{\sigma}(p)}\},\\
 &\hat{c}_{\nu(p-k)}\in \mathcal{C}_{\nu(p-k)}=
 \{ \cl_{\underline{\nu}(p-k)}, \ldots, \cl_{\underline{\nu}(p-k+\Gamma)}\} \cup
 \{ \cu_{\overline{\nu}(1)},\ldots,  \cu_{\overline{\nu}(p-k)}\},\\
 & \hat{c}_{\varsigma(k)}\in  \mathcal{C}_{\varsigma(k)}=
 \{ \cl_{\underline{\varsigma}(k)},\ldots, \cl_{\underline{\varsigma}(k+\Gamma)}\}\cup
 \{ \cu_{\overline{\varsigma}(1)},\ldots,  
 \cu_{\overline{\varsigma}(k)}\}.
\end{align*}
For simplicity of notation, we write $\cl_{\underline{\sigma}(n)}$ instead of $\cl_{\underline{\sigma}(p+\Gamma)}$, when $p+\Gamma>n$.
The same holds for $\cl_{\underline{\nu}(p-k+\Gamma)}$ and
$\cl_{\underline{\varsigma}(k+\Gamma)}$.
Observe that we can assume that $\Gamma\leq n/2$. Indeed, if $\Gamma>n/2$, then it suffices to substitute 
variables~$z_i$ by $1-w_i$, $w_i \in \{0,1\}$, $i\in [n]$. Now
 the constraint $\sum_{i\in [n]} \delta_i \leq \Gamma$ and the costs~$\hat{c}_i$ become $\sum_{i\in[n]} w_i \geq n-\Gamma$
 and $\hat{c}_i=\cl_i+\cd_i(1-w_i)$, respectively.
From Proposition~\ref{propkk} and the above, it follows that $(\alpha,\beta)\in \mathcal{S}_{\pmb{x}}$, where
$ \mathcal{S}_{\pmb{x}}$ is defined as follows:
\begin{align*}
\mathcal{S}_{\pmb{x}}=&
 \{(\alpha,\beta) : \alpha=\hat{c}_{\sigma(p)}, \beta=0,  
    \hat{c}_{\nu(p-k)}\leq \hat{c}_{\sigma(p)},
   \hat{c}_{\sigma(p)}\in \mathcal{C}_{\sigma(p)},
   \hat{c}_{\nu(p-k)}\in \mathcal{C}_{\nu(p-k)}\}\cup\\
   &
   \{(\alpha,\beta) :  \alpha=\hat{c}_{\varsigma(k)}, \beta=\hat{c}_{\nu(p-k)}-\hat{c}_{\varsigma(k)},  
    \hat{c}_{\nu(p-k)}>\hat{c}_{\varsigma(k)},
   \hat{c}_{\nu(p-k)}\in \mathcal{C}_{\nu(p-k)},
   \hat{c}_{\varsigma(k)}\in \mathcal{C}_{\varsigma(k)}\}.\nonumber
\end{align*}
Finally,~(\ref{arecdisc}) becomes
\begin{equation}
\label{arecdisc2}
\begin{array}{llll}
\max \;\; &  \displaystyle p\alpha + (p-k) \beta - \sum_{i\in[n]}  [ \alpha + x_i \beta - \cl_i]_+ \\
& \displaystyle -\sum_{i\in[n]} ( [ \alpha + x_i \beta - \cd_i - \cl_i]_+ - [ \alpha + x_i \beta - \cl_i]_+) \delta_i \\
\text{s.t. } & \displaystyle \sum_{i\in[n]} \delta_i \leq \Gamma  \\
& \delta_i \in \{0,1\} &  i\in[n]  \\
& (\alpha,\beta)\in \mathcal{S}_{\pmb{x}} 
\end{array}
\end{equation}
Accordingly, it now suffices to solve~(\ref{arecdisc2}) for each  $(\alpha,\beta)\in \mathcal{S}_{\pmb{x}}$ and
choose the best of the computed solutions which encodes a worst scenario.
 Solving~(\ref{arecdisc2}) for fixed $(\alpha,\beta)$
can be done in $O(n)$.
Since the cardinality of the sets $\mathcal{C}_{\sigma(p)}$, $\mathcal{C}_{\nu(p-k)}$ and
$\mathcal{C}_{\varsigma(k)}$ is $O(n)$, the cardinality of the set $\mathcal{S}_{\pmb{x}}$ is
$O(n^2)$. This leads to the following theorem:
\begin{theorem}
The problem \textsc{AREC} under scenario set $\cU^d$ can be solved in $O(n^3)$ time.
\end{theorem}

\subsubsection{The recoverable robust problem}
\label{secrecrobdisc}

 We first identify some special cases of \textsc{RREC} which are polynomially solvable.
\begin{observation}
The following special cases of \textsc{RREC} under $\cU^d$ are polynomially solvable:
	\begin{enumerate}
		\item[(i)]  if $k=0$, then \textsc{RREC} is solvable in $O(n^2)$ time,
		\item[(ii)]  if $\Gamma=n$, then \textsc{RREC} is solvable in $O((p-k+1)n^2)$ time,
		\item[(iii)] if $k\geq \Gamma$ and $C_i=0$, $i\in [n]$, then \textsc{RREC} is solvable in $O(n)$ time.
	\end{enumerate}
\end{observation}
\begin{proof}
\begin{enumerate}
\item[(i)] If $k=0$, then \textsc{RREC} is equivalent to the \textsc{MinMax} problem under scenario set $\cU^d$ with the cost intervals $[C_i+\cl_i, C_i+\cu_i]$, $i\in [n]$. This problem is solvable in $O(n^2)$ according to the results obtained in~\cite{BS03}.

\item[(ii)] If $\Gamma=n$, then \textsc{RREC} can be reduced to the recoverable robust problem under scenario set $\cU^I$, which can be solved in $\mathcal{O}((p-k+1)n^2)$ time~\cite{KZ15b}.

\item[(iii)] Consider the case $k\geq \Gamma$.
 Let $\pmb{x}^*\in \Phi$ be an optimal solution to the \textsc{Selection} problem for the costs~$\cl_i$, $i\in[n]$,
 and $\pmb{\delta}\in\{0,1\}^n$ stands for  any vector  that
 encodes scenario $\cl_i+\delta_i\cd_i$, $i\in[n]$,  $\sum_{i\in[n]}\delta_i\leq\Gamma$.
 Let $\pmb{y}$ be an optimal solution to the \textsc{Selection} problem
 with respect to~$\cl_i+\delta_i\cd_i$, $i\in[n]$. 
Now  $\pmb{x}^*$ and $\pmb{y}$ have at least $p-\Gamma$ elements in common,
which is due to the fact that $\sum_{i\in[n]}\delta_i\leq\Gamma$.
 Since $k\geq \Gamma$,  $\pmb{x}^*$ can be recovered to~$\pmb{y}$.
 Hence, no better
solution can exist.
\end{enumerate}
% ~To see that point~3 is true,  let $\pmb{x}\in \Phi$ be an optimal solution to the \textsc{Selection} problem for the costs~$c_i$, $i\in[n]$. Let $\pmb{z}\in \Phi$ be an optimal solution to the \textsc{Selection} problem under a worst scenario $\pmb{c}\in \cU^d$ for $\pmb{x}$ (i.e. $\pmb{c}$ is an optimal solution to the adversarial problem for $\pmb{x}$). Then $X_{\pmb{x}}$ and $X_{\pmb{z}}$ have at least $p-\Gamma$ items in common.
%%Then  the $p$ smallest elements with respect to~$c_i$, $i\in[n]$ and
%%the $p$ smallest elements with respect to~$c_i+z_id_i$, $i\in[n]$,
%%have at least $p-\Gamma$ elements in common.
%Since $k\geq \Gamma$,  $\pmb{x}$ can be recovered~$\pmb{z}$.
%%that is optimal with respect to~$c_i+z_id_i$, $i\in[n]$. Hence, no better
%%solution can exist.
\end{proof}

We will now construct a compact MIP formulation for the general \textsc{RREC} problem under scenario set $\cU^d$. In order to do this we will use the formulation~(\ref{arecdisc2}). Observe  that  the sets 
 $\mathcal{C}_{\nu(p-k)}$ and $\mathcal{C}_{\varsigma(k)}$, defined in Section~\ref{secadvdisc}, depend on a fixed solution $\pmb{x}$ (the set $\mathcal{C}_{\sigma(p)}$ does not depend on $\pmb{x}$).
% of possible candidates for
%$\hat{c}_{\nu(p-k)}$ and $\hat{c}_{\varsigma(k)}$, respectively, and consequently the set~$\mathcal{S}_X$ 
%depend on a fixed  solution~$X$ and permutations of $X$ and   $[n]\setminus X$.
We now define $\pmb{x}$-independent sets
of possible values
of the $(p-k)$th  smallest  element in $X_{\pmb{x}}$ and the $k$th  smallest element in $\overline{X}_{\pmb{x}}$ under any scenario in $\cU^d$ by
\begin{align*}
 &\hat{c}_{\sigma(p-k)}\in \mathcal{C}_{\sigma(p-k)}=
 \{ \cl_{\underline{\sigma}(p-k)},\ldots, \cl_{\underline{\sigma}(n-k+\Gamma)}\}\cup
 \{  \cu_{\overline{\sigma}(1)},\ldots,   \cu_{\overline{\sigma}(n-k)}\},\\
 &\hat{c}_{\sigma(k)}\in \mathcal{C}_{\sigma(k)}=
 \{ \cl_{\underline{\sigma}(k)},\ldots, \cl_{\underline{\sigma}(k+p+\Gamma)}\}\cup
 \{  \cu_{\overline{\sigma}(1)},\ldots,   \cu_{\overline{\sigma}(k+p)}\}.
 \end{align*}
 Again, from Proposition~\ref{propkk} and the above, we get a new set of relevant values
 of~$\alpha$ and~$\beta$
  \begin{align*}
\mathcal{S}=&
 \{(\alpha,\beta) : \alpha=\hat{c}_{\sigma(p)}, \beta=0,  
    \hat{c}_{\sigma(p-k)}\leq \hat{c}_{\sigma(p)},
   \hat{c}_{\sigma(p)}\in \mathcal{C}_{\sigma(p)},
   \hat{c}_{\sigma(p-k)}\in \mathcal{C}_{\sigma(p-k)}\}\cup\\
   &
   \{(\alpha,\beta) :  \alpha=\hat{c}_{\sigma(k)}, \beta=\hat{c}_{\sigma(p-k)}-\hat{c}_{\sigma(k)},  
    \hat{c}_{\sigma(p-k)}>\hat{c}_{\sigma(k)},
   \hat{c}_{\sigma(p-k)}\in \mathcal{C}_{\sigma(p-k)},
   \hat{c}_{\sigma(k)}\in \mathcal{C}_{\sigma(k)}\}.\nonumber
\end{align*}
Obviously 
 $\mathcal{C}_{\nu(p-k)}\subseteq   \mathcal{C}_{\sigma(p-k)}$,
 $\mathcal{C}_{\varsigma(k)}\subseteq  \mathcal{C}_{\sigma(k)}$,
$\mathcal{S}_{\pmb{x}}\subseteq \mathcal{S}$ for any $\pmb{x}\in \Phi$ and the cardinality of~$\mathcal{S}$ remains~$\mathcal{O}(n^2)$. Let us represent the adversarial problem~(\ref{arecdisc2}) as follows:
\begin{equation}
\label{arecdisc3}
\begin{array}{llll}
\displaystyle \max_{(\alpha, \beta)\in \mathcal{S}} \max_{\pmb{\delta}}  &  \displaystyle p\alpha + (p-k) \beta - \sum_{i\in[n]}  [ \alpha + x_i \beta -\cl_i]_+ + \\
& \displaystyle  \sum_{i\in[n]} ([ \alpha + x_i \beta -\cl_i]_+-[ \alpha + x_i \beta - \cd_i -\cl_i]_+) \delta_i \\
\text{s.t. } & \displaystyle \sum_{i\in[n]} \delta_i \leq\Gamma  \\
& \delta_i \in \{0,1\} &  i\in[n]  
\end{array}
\end{equation}
Dualizing the inner \textsc{Selection} problem in~(\ref{arecdisc3}),  we get:
\begin{align*}
\max_{(\alpha,\beta) \in\mathcal{S}} \min_{\pi,\pmb{\rho}} \ &\Gamma \pi + \sum_{i\in[n]} \rho_i + p\alpha + (p-k)\beta - \sum_{i\in[n]} [ \alpha + x_i \beta - \cl_i]_+ \\
\text{s.t. } & \pi + \rho_i \ge [\alpha + x_i \beta - \cl_i]_+ - [\alpha + x_i\beta - \cl_i - \cd_i]_+ &  i\in[n] \\
& \rho_i \ge 0 &  i\in[n]\\
&\pi\geq 0&
\end{align*}
 For every pair $(\alpha^\ell,\beta^\ell)\in\mathcal{S}$, we introduce a set of variables $\pi^\ell$, $\pmb{\rho}^\ell$, $\ell\in [S]$ with $S = |\mathcal{S}|$. The \textsc{RREC} problem under scenario set $\cU^d$ is then equivalent to the following mathematical programming problem:
\begin{align*}
\min\ & \lambda \\
\text{s.t. } & \lambda
 \ge
 \sum_{i\in[n]} C_ix_i+
  \Gamma \pi^\ell + \sum_{i\in[n]} \rho^\ell_i + p\alpha^\ell + (p-k)\beta^\ell - \sum_{i\in[n]} [ \alpha^\ell + x_i \beta^\ell - \cl_i ]_+ &  \ell \in [S] \\
& \pi^\ell + \rho^\ell_i \ge [\alpha^\ell + x_i \beta^\ell - \cl_i]_+ - [\alpha^\ell + x_i\beta^\ell - \cl_i - \cd_i]_+ &  i\in[n], \ell\in[S] \\
& \sum_{i\in[n]} x_i = p \\
& x_i\in\{0,1\} &  i\in[n] \\
& \rho^\ell_i \ge 0 &  \ell\in[S],i\in[n]\\
& \pi^\ell \ge 0 &  \ell\in[S]
\end{align*}
Finally, we can linearize all the nonlinear terms $[a+bx_i]_+=[a]_{+} + ([a+b]_{+}-[a]_+)x_i$, where $a,b$ are constant. In consequence, we obtain a compact MIP formulation for \textsc{RREC}, with $O(n^3)$ variables and $O(n^3)$ constraints.

We now present an approximation algorithm, which can be applied for larger problem instances.
Suppose that $\cl_i\geq \alpha \cu_i$ for each item $i\in [n]$, where $\alpha\in (0,1]$ is a constant.  
This inequality means that for each item $i$ the nominal cost~$\cl_i$ is positive and $\cu_i$  is at most $1/\alpha$ greater than $\cl_i$. It is reasonable to assume that this condition will be true in many practical applications for not very large value of $1/\alpha$. Consider scenario set $\cU'=\{(\cl_i)_{i\in [n]}\}$, so $\cU'$ contains only one scenario composed of the nominal item costs. Let $\hat{\pmb{x}}$ be an optimal solution to the \textsc{RREC} problem under $\cU'$. This solution can be computed in polynomial time~\cite{KZ15b}. Using the same reasoning as in~\cite{HKZ16a}, one can show that the cost of $\hat{\pmb{x}}$ is at most $1/\alpha$ greater than the optimum. Hence there is an $1/\alpha$ approximation algorithm for \textsc{RREC} under scenario set $\cU^d$.

\subsection{Two-Stage Robust Selection}

\subsubsection{The adversarial problem}
 
Let $\pmb{x}\in \Phi_1$ be a fixed first stage solution, with $|X_{\pmb{x}}|=p_1$.
Using the same reasoning as in Section~\ref{sec32contia}, we can represent the \textsc{A2ST} problem as a special case of \textsc{AREC} with $\tilde{p}=p-p_1$ and $k=\tilde{p}$. In this case $\beta=0$ in the formulation~(\ref{arecdisc0}) and there are only $O(n)$ candidate values for $\alpha$. Hence the problem can be solved in $O(n^2)$ time under scenario set~$\cU^d$. 
 
%\new{This can also be seen from the following formulations.} 
We now show some additional properties of the adversarial problem, which will then be
% will next 
used to solve the more general \textsc{R2ST} problem.
% \todo[author=Adam, size=\footnotesize]{I have change this sentence} 
By dualizing the linear programming relaxation of the incremental problem (\ref{inc2stc}), we find the following MIP formulation for the adversarial problem \textsc{A2ST}:
\begin{equation}
\label{advdisc}
\begin{array}{lllll}
\max\ & (p - \displaystyle \sum_{i\in[n]} x_i) \alpha - \sum_{i\in[n]} (1-x_i) \gamma_i \\
\text{s.t. } & \alpha \le \gamma_i + \cl_i + \cd_i\delta_i & i\in[n] \\
& \displaystyle \sum_{i\in[n]} \delta_i \le \Gamma \\
& \gamma_i \ge 0 & i\in [n]\\
& \delta_i \in\{0,1\} & i\in [n]
\end{array}
\end{equation}
The following lemma characterizes an optimal solution to~(\ref{advdisc}):
\begin{lemma}\label{lem22}
There is an optimal solution to~\eqref{advdisc} in which $\alpha=0$, $\alpha= \cl_j$ or $\alpha=\cu_j$ for some $j\in[n]$.
\end{lemma}
\begin{proof}
Let us fix $\pmb{\delta}$ in~(\ref{advdisc}). Then the resulting linear program with additional slack variables is 
\begin{align*}
\max\ & (p - \displaystyle \sum_{i\in[n]} x_i) \alpha - \sum_{i\in[n]} (1-x_i) \gamma_i \\
\text{s.t. } & \alpha + \epsilon_i - \gamma_i = \cl_i + \cd_i\delta_i & i\in[n] \\
& \alpha \ge 0 \\
& \epsilon_i, \gamma_i \ge 0 & i\in [n]
\end{align*}
Note that we only consider nonnegative values of dual variable~$\alpha$
associated with  the cardinality constraint  in (\ref{inc2stc}), since replacing this constraint in (\ref{inc2stc})
by  $\sum_{i\in[n]} (y_i + x_i )\geq p$ does not change the set of optimal solutions.
The problem has $2n+1$ variables and $n$ constraints. If $\alpha$ is a non-basis variable in an optimal solution, then $\alpha=0$. So, let us assume that $\alpha$ is a basis variable. As there are $n-1$ remaining basis variables, there is at least one $j\in[n]$ where both $\epsilon_j$ and $\gamma_j$ are non-basis variables. Hence, $\alpha = \cl_j + \cd_j\delta_j$ and the lemma follows since $\delta_j\in \{0,1\}$.
\end{proof}
Define $\mathcal{S} = \{ 0 \} \cup \{ \cl_i : i\in[n] \} \cup \{ \cu_i : i\in[n] \}$
and write $\mathcal{S} = \{ \alpha^1, \ldots, \alpha^S\}$ with $S=|\mathcal{S}|\in O(n)$.
Using Lemma~\ref{lem22}, problem~(\ref{advdisc}) is then equivalent to
\begin{align*}
\max_{\alpha\in\mathcal{S}} \max_{\pmb{\delta},\pmb{\gamma}} \ & (p - \sum_{i\in[n]} x_i) \alpha - \sum_{i\in[n]} (1-x_i) \gamma_i \\
\text{s.t. } & \gamma_i \ge \alpha - \cl_i - \cd_i\delta_i \\
& \sum_{i\in[n]} \delta_i \le \Gamma \\
& \gamma_i \ge 0 & i\in [n]\\
& \delta_i \in\{0,1\} & i\in[n]
\end{align*}
Let $(\pmb{\delta}^*,\pmb{\gamma}^*)$ be an optimal solution to the inner maximization problem. Note that we can assume that 
if $\delta^*_i = 0$, then $\gamma^*_i = [\alpha-\cl_i]^+$  and if $\delta^*_i=1$, then $\gamma^*_i = [\alpha-\cl_i-\cd_i]^+$. Hence, the inner problem is equivalent to
\begin{equation}
\label{advdisc2}
\begin{array}{lllll}
\max\ &(p-\sum_{i\in[n]} x_i))\alpha - \sum_{i\in[n]} (1-x_i) [\alpha-\cl_i]^+ \\
 & + (1-x_i)([\alpha-\cl_i-\cd_i]^+ - [\alpha-\cl_i]^+)\delta_i \\
\text{s.t. } & \sum_{i\in[n]} \delta_i \le \Gamma \\
& \delta_i\in\{0,1\} & i\in[n]
\end{array}
\end{equation}
As this is the \textsc{Selection} problem, we state the following result:

\begin{theorem}
The problem A2ST under scenario set $\cU^d$ can be solved in $O(n^2)$ time.
\end{theorem}

\subsubsection{The two-stage robust problem} 

The \textsc{R2ST} problem is polynomially solvable when $\Gamma\geq n$. Scenario set $\cU^d$ can be then replaced by $\cU^I$ and the problem is solvable in $O(n)$ time~\cite{KZ15b}.

We now present a compact MIP formulation for \textsc{R2ST} under scenario set $\cU^d$. 
We can use the dual of the linear relaxation of $\eqref{advdisc2}$ and the set $\mathcal{S}$ of candidate values for $\alpha$ to arrive at the following formulation:
\begin{align*}
\min\ & \lambda \\
\text{s.t. } & \lambda \ge \sum_{i\in[n]} C_ix_i + (p-\sum_{i\in[n]} x_i) \alpha^\ell - \sum_{i\in[n]} (1-x_i) [\alpha^\ell-\cl_i]^+ + \Gamma \pi^\ell + \sum_{i\in[n]} \rho^\ell_i & \ell\in[S] \\
& \sum_{i\in[n]} x_i \le p \\
& \pi^\ell + \rho^\ell_i \ge (1-x_i) ([\alpha^\ell-\cl_i]^+ - [\alpha^\ell-\cl_i-\cd_i]^+) & i\in[n], \ell\in[S]\\
& x_i \in\{0,1\} & i\in [n]\\
& \pi^\ell \ge 0 & \ell \in [S] \\
& \rho^\ell_i \ge 0 & i\in[n], \ell\in[S]
\end{align*}
This formulation has $O(n^2)$ variables and $O(n^2)$ constraints.

% \todo[author=Adam, size=\footnotesize]{Appr. algorithm. Please verify} 
We now propose a fast approximation algorithm for the problem. The idea is the same as for the robust recoverable problem (see Section~\ref{secrecrobdisc} and also~\cite{HKZ16a}). Let us fix scenario $\underline{\pmb{c}}=(\cl_i)_{i\in [n]}\in \cU^d$, which is composed of the nominal item costs.  Consider the following problem:
\begin{equation}
\label{appr1}
\min_{\pmb{x}\in \Phi_1} (\pmb{C}\pmb{x}+\min_{\pmb{y}\in \Phi_{\pmb{x}}} \underline{\pmb{c}}\pmb{y}).
\end{equation}
Problem~(\ref{appr1}) can be solved in $O(n)$ time in the following way. Let $c_i'=\min\{C_i, \cl_i\}$ for each $i\in [n]$ and let $\pmb{x}\in\Phi$ be an optimal solution to the \textsc{Selection} problem for the costs $c_i'$. The optimal solution $\hat{\pmb{x}}\in \Phi_1$ to~(\ref{appr1}) is then formed by fixing $\hat{x}_i=1$ if $x_i=1$ and $c_i'=C_i$, and $\hat{x}_i=0$ otherwise. We now prove the following result:
\begin{proposition}
\label{propappr}
	Let $\cl_i\geq \alpha \cu_i$ for each $i\in [n]$ and $\alpha\in (0,1]$.  Let $\hat{\pmb{x}} \in \Phi_1$ be an optimal solution to~(\ref{appr1}). Then $\hat{\pmb{x}}$ is an $\frac{1}{\alpha}$- approximate solution to \textsc{R2ST}.
\end{proposition}
\begin{proof}
	Let $\pmb{x}^*$ be an optimal solution to \textsc{R2ST} with the objective value $OPT$. It holds
	$$OPT=\pmb{C}\pmb{x}^*+\max_{\pmb{c}\in \cU^d}\min_{\pmb{y}\in \Phi_{\pmb{x}^*}} \pmb{c}\pmb{y}=\pmb{C}\pmb{x}^*+\pmb{c}^*\pmb{y}^*\geq \pmb{C}\pmb{x}^*+\underline{\pmb{c}}\pmb{y}^*,$$
	because $\cl_i\leq c^*_i$ for each $i\in [n]$. Since $\hat{\pmb{x}}$ is an optimal solution to~(\ref{appr1}) we get
	$$\pmb{C}\pmb{x}^*+\underline{\pmb{c}}\pmb{y}^*\geq \pmb{C}\hat{\pmb{x}}+\underline{\pmb{c}}\hat{\pmb{y}},$$
	where $\hat{\pmb{y}}=\min_{\pmb{y}\in \Phi_{\hat{\pmb{x}}}} \underline{\pmb{c}}\pmb{y}$.
	By the assumption that $\cl_i\geq \alpha \cu_i$ for each $i\in [n]$, we obtain
	$$\pmb{C}\hat{\pmb{x}}+\underline{\pmb{c}}\hat{\pmb{y}}\geq \pmb{C}\hat{\pmb{x}} + \alpha \overline{\pmb{c}}\hat{\pmb{y}}.$$
	Finally, as $\alpha\in (0,1]$
	$$ \pmb{C}\hat{\pmb{x}} + \alpha \overline{\pmb{c}}\hat{\pmb{y}} \geq \alpha ( \pmb{C}\hat{\pmb{x}} +  \overline{\pmb{c}}\hat{\pmb{y}})\geq \alpha (\pmb{C}\hat{\pmb{x}}+\max_{\pmb{c}\in \cU^d} \pmb{c}\hat{\pmb{y}})\geq \alpha (\pmb{C}\hat{\pmb{x}}+\max_{\pmb{c}\in \cU^d}\min_{\pmb{y}\in \Phi_{\hat{\pmb{x}}}} \pmb{c}\pmb{y})$$
	and the proposition follows.
\end{proof}

\section{Conclusion}

\label{sec5}
The \textsc{Selection} problem is one of the main objects of study for the complexity of robust optimization problems. While the robust counterpart of most combinatorial optimization problems is NP-hard, 
its simple structure allows in many cases to construct efficient polynomial algorithms. As an example, the \textsc{MinMax-Regret Selection} problem was the first \textsc{MinMax-Regret} problem for which polynomial time solvability could be proved \cite{A01}.

In this paper we continue this line of research by considering
recoverable and two-stage robust problems combined with discrete and budgeted uncertainty sets. All four problem combinations have not been analyzed before, and little is known about other problems of this kind.

We showed that the continuous uncertainty problem variants allow polynomial-time solution algorithms, based on solving a set of linear programs. Additionally, we derived strongly polynomial combinatorial algorithms for the adversarial subproblems and discussed ways to preprocess instances.

For the discrete uncertainty case, we also presented strongly polynomial combinatorial algorithms for the adversarial problems, and constructed mixed-integer programming formulations of polynomial size. It remains an open problem to analyze in future research if the problems with discrete uncertainty are NP-hard or also allow for polynomial-time solution algorithms.

Further research includes the application of our setting to other combinatorial optimization problems, such as \textsc{Spanning Tree} or \textsc{Shortest Path}.

\appendix

\section{An Improved Algorithm for  \textsc{AREC}}
\label{sec:appendix}

We now discuss how the $O(n^2)$ result from Section~\ref{secARECc} can be further improved to $O(n\log n)$. Using the same arguments as before, this result then also applies to \textsc{A2ST}.

Let a solution $\pmb{x}\in\Phi$ and two numbers $L_{\pmb{x}}$,$L_{\overline{\pmb{x}}}$ be given. We call a solution $\pmb{c}$ of problem  \textsc{AREC}  $(L_{\pmb{x}},L_{\overline{\pmb{x}}})-$\emph{compatible}, if 
\[c_i = \max\{\cl_i,\min\{\cu_i,L_{\pmb{x}}\}\} \quad \forall i\in X_{\pmb{x}} \]
\[c_i = \max\{\cl_i,\min\{\cu_i,L_{\overline{\pmb{x}}}\}\} \quad \forall i\in \overline{X}_{\pmb{x}} \]
Note that for each tuple $(L_{\pmb{x}},L_{\overline{\pmb{x}}})$ there is a unique solution (scenario in $\cU^c$) which is $(L_{\pmb{x}},L_{\overline{\pmb{x}}})-$\emph{compatible}. We denote this solution by $\pmb{c}(L_{\pmb{x}},L_{\overline{\pmb{x}}})$.
\begin{lemma}
Each solution $\pmb{c}$ of  \textsc{AREC}  can be transformed to a solution $\pmb{c}'$ which is $(L_{\pmb{x}},L_{\overline{\pmb{x}}})-$\emph{compatible} for some $L_{\pmb{x}},L_{\overline{\pmb{x}}}$ and has no lower objective value.
%\todo[author=Adam,  size=\footnotesize]{ This Lemma is a consequence of Proposition~2. The same transformations as in Proposition~2 gives this result. I suggest to remove the proof}
%\todo[author=Pawel, size=\footnotesize]{I agree with Adam}
\label{lem_transform_to_compatible}
\end{lemma}
\begin{proof}
The proof contains similar arguments as the proof of Proposition~\ref{prop2}. Shifting the budget from items which are above or at the corresponding levels to items which are not yet at its cost upper bound but below the corresponding levels does not decrease the objective function.
\end{proof}

%Given an arbitrary solution $\pmb{c}$ of  \textsc{AREC}. Denote by $I_1\subseteq[n]$ the $p-k$ smallest items of $X_{\pmb{x}}$ and by $I_2\subseteq[n]$ the $k$ smallest items of $[n]\backslash I_1$. Note that the solution of the incremental problem is given by the set $I= I_1 \cup I_2$. Let $L_{\pmb{x}}'$ be the size of the largest item of $I \cap X_{\pmb{x}}$ and $L_{\overline{\pmb{x}}}'$ the size of the largest item of $I \cap \overline{X}_{\pmb{x}}$ (if $I \cap \overline{X}_{\pmb{x}}  = \emptyset$, we set $L_{\overline{\pmb{x}}}'$ to the smallest element of $\overline{X}_{\pmb{x}}$). 
%
%All budget which is invested into items with $c_i> \cl_i$ and above its corresponding level (either $L_{\pmb{x}}'$ if $i\in X_{\pmb{x}}$ or $L_{\overline{\pmb{x}}}'$ if $i \in \overline{X}_{\pmb{x}}$), has no influences on the objective value of the solution. Hence, it could be either used to increase the cost of items from $I$ if these items are not already at its upper bound costs or it could be simply removed. In both cases, the objective value does not decrease.
%
%For all item from $I$ which are not at its upper bound cost and not at its corresponding level, we can transfer the budget which is invested into more expensive items from $I$, which are not at is nominal cost, to these items. In this way the objective value of the solution stays constant but the corresponding levels $L_{\pmb{x}}'$ and $L_{\overline{\pmb{x}}}'$ may decrease. In conclusion, we end up with an solution which is $(L_{\pmb{x}}',L_{\overline{\pmb{x}}}')-$compatible solution.

\noindent
Denote by $\pmb{c}_{\pmb{x}}$ ($\pmb{c}_{\overline{\pmb{x}}}$) the projection of $\pmb{c}$ onto $X_{\pmb{x}}$ ($\overline{X}_{\pmb{x}}$).
The objective value, $F(\pmb{c})$, of solution $\pmb{c}$ for problem  \textsc{AREC}  can be expressed as follows:

\[F(\pmb{c})= \min_{j\in \{p-k,p-k+1,\dots,p\}}  \langle\pmb{c}_{\pmb{x}}\rangle_{(j)} + \langle\pmb{c}_{\overline{\pmb{x}}}\rangle_{(p-j)}\]

\noindent
where $\langle\pmb{v}\rangle_{(l)}$ denotes the sum of the $l$ smallest entries of vector $\pmb{v}$ ($\langle\pmb{v}\rangle_{0}:=0$).

\begin{lemma}
$F$ is a concave function on $\mathbb{R}^n_+$.
\label{lem_f_concave}
\end{lemma}
\begin{proof}
Note that $h(\pmb{v})=\langle\pmb{v}\rangle_{(l)}$ is a concave function on $\mathbb{R}^n_+$. Hence, since the sum of concave functions is concave, $f_j(\pmb{c}) := \langle\pmb{c}_{\pmb{x}}\rangle_{(j)} + \langle\pmb{c}_{\overline{\pmb{x}}}\rangle_{(p-j)}$ is a concave function. Since the minimum of concave functions is again a concave function, it follows that $F:= \min_{j\in \{p-k,p-k+1,\dots,p\}} f_j$ is a concave function.
\end{proof}
\noindent
For two levels ($L_{\pmb{x}},L_{\overline{\pmb{x}}}$) we define the set of \emph{active items}:
\[A_{\pmb{x}}(L_{\pmb{x}},L_{\overline{\pmb{x}}}) := \{i \in X_{\pmb{x}}: c(L_{\pmb{x}},L_{\overline{\pmb{x}}})_i  = L_{\pmb{x}}, c(L_{\pmb{x}},L_{\overline{\pmb{x}}})_i < \cu_i\},\]
\[A_{\overline{\pmb{x}}}(L_{\pmb{x}},L_{\overline{\pmb{x}}}) := \{i \in \overline{X}_{\pmb{x}}: c(L_{\pmb{x}},L_{\overline{\pmb{x}}})_i = L_{\overline{\pmb{x}}}, c(L_{\pmb{x}},L_{\overline{\pmb{x}}})_i < \cu_i\}.\]

Recall that in the incremental problem, exactly $p$ items need to be selected. After the $p-k$ cheapest items are chosen from $X_{\pmb{x}}$, the $k$ cheapest items are chosen from the remaining items. If two items have the same cost, we differentiate between two tie-breaking rules: Prefer items from $X_{\pmb{x}}$ (rule~1) and prefer items from $\overline{X}_{\pmb{x}}$ (rule~2). We solve the incremental problem with cost $\pmb{c}(L_{\pmb{x}},L_{\overline{\pmb{x}}})$ with tie-breaking rule~2. The set of \emph{selected items} $I_{\pmb{x}}(L_{\pmb{x}},L_{\overline{\pmb{x}}})$ is defined as the set of all items from $X_{\pmb{x}}$ which are part of the incremental solution. Analogously, we define the set $I_{\overline{\pmb{x}}}(L_{\pmb{x}},L_{\overline{\pmb{x}}})$ as the set of all items from $\overline{X}_{\pmb{x}}$ which are part of the incremental solution if tie-breaking rule~1 is used. %\todo[author=Adam, size=\footnotesize]{So, which tie-breaking rule is used?}
Further, we define the set $I(L_{\pmb{x}},L_{\overline{\pmb{x}}}) \subseteq [n]$ as the complete solution of the incremental problem. For the definition of this set, the choice of the tie-breaking rule is not important (without loss of generality we may assume that rule~1 is used). 

Assume we are given a small budget $\epsilon$ which we can invest to transform $\pmb{c}(L_{\pmb{x}},L_{\overline{\pmb{x}}})$ to a new solution $\pmb{c}(L_{\pmb{x}}',L_{\overline{\pmb{x}}}')$ with $L_{\pmb{x}} \leq L_{\pmb{x}}'$ and $L_{\overline{\pmb{x}}} \leq L_{\overline{\pmb{x}}}'$. Further assume that the sets of active items and the sets of selected items are the same for $(L_{\pmb{x}},L_{\overline{\pmb{x}}})$ and $(L_{\pmb{x}}',L_{\overline{\pmb{x}}}')$, and that we decide to keep $L_{\overline{\pmb{x}}}'=L_{\overline{\pmb{x}}}$ (the case that $L_{\pmb{x}} = L_{\pmb{x}}'$ is analogous). To increase $L_{\pmb{x}}$ we have to equally distribute the budget on all items from $A_{\pmb{x}}(L_{\pmb{x}},L_{\overline{\pmb{x}}})$. Hence, we get that
\[L_{\pmb{x}}' = L_{\pmb{x}} + \frac{\epsilon}{|A_{\pmb{x}}(L_{\pmb{x}},L_{\overline{\pmb{x}}})|}.\]
The objective value of the solution increases only for the selected items which were increased. Hence, we have that 
\[F(\pmb{c}(L_{\pmb{x}}',L_{\overline{\pmb{x}}}'))= F(\pmb{c}(L_{\pmb{x}},L_{\overline{\pmb{x}}})) + \epsilon\frac{|I_{\pmb{x}}(L_{\pmb{x}},L_{\overline{\pmb{x}}})|}{|A_{\pmb{x}}(L_{\pmb{x}},L_{\overline{\pmb{x}}})|}.\]
We see that the gain in the objective function is determined by the ratio of selected and active items. We denote these ratios in the following as
\begin{align*}
R_{\pmb{x}}(L_{\pmb{x}},L_{\overline{\pmb{x}}}) &:= \frac{|I_{\pmb{x}}(L_{\pmb{x}},L_{\overline{\pmb{x}}}) \cap A_{\pmb{x}}(L_{\pmb{x}},L_{\overline{\pmb{x}}})|}{|A_{\pmb{x}}(L_{\pmb{x}},L_{\overline{\pmb{x}}})|} \\
R_{\overline{\pmb{x}}}(L_{\pmb{x}},L_{\overline{\pmb{x}}}) &:= \frac{|I_{\overline{\pmb{x}}}(L_{\pmb{x}},L_{\overline{\pmb{x}}}) \cap A_{\overline{\pmb{x}}}(L_{\pmb{x}},L_{\overline{\pmb{x}}})|}{|A_{\overline{\pmb{x}}}(L_{\pmb{x}},L_{\overline{\pmb{x}}})|} \\
R(L_{\pmb{x}},L_{\overline{\pmb{x}}}) &:= \frac{| I(L_{\pmb{x}},L_{\overline{\pmb{x}}}) \cap (A_{\pmb{x}}(L_{\pmb{x}},L_{\overline{\pmb{x}}}) \cup A_{\overline{\pmb{x}}}(L_{\pmb{x}},L_{\overline{\pmb{x}}}))|}{|A_{\pmb{x}}(L_{\pmb{x}},L_{\overline{\pmb{x}}}) \cup A_{\overline{\pmb{x}}}(L_{\pmb{x}},L_{\overline{\pmb{x}}})|}. 
\end{align*}
Ratio $R_{\pmb{x}}(L_{\pmb{x}},L_{\overline{\pmb{x}}})$ ($R_{\overline{\pmb{x}}}(L_{\pmb{x}},L_{\overline{\pmb{x}}})$) defines how efficient it is to increase $L_{\pmb{x}}$ ($L_{\overline{\pmb{x}}}$) and the ratio $R(L_{\pmb{x}},L_{\overline{\pmb{x}}})$ computes the efficiency of increasing $L_{\pmb{x}}$ and $L_{\overline{\pmb{x}}}$ simultaneously.  

The algorithm performs a sequence of greedy update steps which are locally optimal. We prove later that the so found solution is indeed optimal. The algorithm starts with solution $\pmb{c}^0 := \underline{\pmb{c}}$. The initial levels $L_{\pmb{x}}^0$ and $L_{\overline{\pmb{x}}}^0$ are equal to the smallest value of $\pmb{c}_{\pmb{x}}^0$ and $\pmb{c}_{\overline{\pmb{x}}}^0$. The algorithm computes $R_{\pmb{x}}:=R_{\pmb{x}}(L_{\pmb{x}}^0,L_{\overline{\pmb{x}}}^0)$, $R_{\overline{\pmb{x}}}:=R_{\overline{\pmb{x}}}(L_{\pmb{x}}^0,L_{\overline{\pmb{x}}}^0)$, and $R:=R(L_{\pmb{x}}^0,L_{\overline{\pmb{x}}}^0)$.

The idea is to change the values of $\pmb{c}^0$ to $\pmb{c}^1$ such that either $L_{\pmb{x}}^0$ or $L_{\overline{\pmb{x}}}^0$ (or both) are increased to $L_{\pmb{x}}^1$ or $L_{\overline{\pmb{x}}}^1$, depending on which update has the highest efficiency. 

First, assume that $L_{\pmb{x}}^0 \neq L_{\overline{\pmb{x}}}^0$. Depending on whether $R_{\pmb{x}}$ or $R_{\overline{\pmb{x}}}$ is higher, we increase either $L_{\pmb{x}}^0$ or $L_{\overline{\pmb{x}}}^0$, until one of the efficiency ratio changes or if the complete budget $\Gamma$ is spent. A change of the efficiency ratio means that either the set of active or the set of selected items changes. We define the set of \emph{interesting levels} as $\mathcal{L}:=\{\cl_1,\dots,\cl_n,\} \cup \{\cu_1,\dots,\cu_n,\}$. Note that the set of active items can only change at levels in $\mathcal{L}$. The set of selected items can only change at interesting levels and if both levels $L_{\pmb{x}}$ and $L_{\overline{\pmb{x}}}$ become equal. Therefore, the amount for which we can increase either $L_{\pmb{x}}^0$ or $L_{\overline{\pmb{x}}}^0$ before one of the efficiency ratio changes is strictly greater zero. 

%Assume that a small amount $\epsilon$ of the budget is invested to increase $L_{\pmb{x}}$ where $\epsilon$ is so small that neither $R_{\pmb{x}}$ nor $R_{\overline{\pmb{x}}}$ changes (this is possible since $L_{\pmb{x}} \neq L_{\overline{\pmb{x}}}$ and due to the definition of $I_{\pmb{x}}(\pmb{c})$ and the used tie-breaking rule). In this case, the objective function $F$ increases exactly by $\epsilon R_{\pmb{x}}$. Increasing $L_{\overline{\pmb{x}}}$ leads to an increase of $\epsilon R_{\overline{\pmb{x}}}$.  We observe that the efficiency ratios $R_{\pmb{x}}$ and $R_{\overline{\pmb{x}}}$ are clear indicators on how to locally improve the solution. 

Next, consider the case that $L_{\pmb{x}}^0 = L_{\overline{\pmb{x}}}^0$. In this case, it can happen that increasing $L_{\pmb{x}}^0$ and $L_{\overline{\pmb{x}}}^0$ simultaneously is more efficient than increasing only one of these two levels. Hence, if $R \geq \max\{R_{\pmb{x}},R_{\overline{\pmb{x}}}\}$, we change $\pmb{c}^0$ to $\pmb{c}^1$ such that both $L_{\pmb{x}}^0$ and $L_{\overline{\pmb{x}}}^0$ are increased. Otherwise we increase only $L_{\pmb{x}}^0$ or $L_{\overline{\pmb{x}}}$, as described before. Investing a budget of $\epsilon$ in a simultaneous update of $L_{\pmb{x}}^0$ and $L_{\overline{\pmb{x}}}$ increases the objective function by $\epsilon R$.

We present in Figures~\ref{fig_r_x}, \ref{fig_r_y}, and \ref{fig_r} all three possible situations. The sample instance consists of $10$ items, the left items are $X_{\pmb{x}}$ and the right items are $\overline{X}_{\pmb{x}}$. Further, $p=5$ and $k=2$. The bars represent the actual costs of each item and the boxes visualize the upper bound $\cu_i$ for the cost of each item. The items which are active have a grid pattern, the other items diagonal lines. In Figure~\ref{fig_r_x}, the best improvement is done by increasing $L_{\pmb{x}}$. In Figure~\ref{fig_r_y}, the best improvement is done by increasing $L_{\overline{\pmb{x}}}$. In Figure~\ref{fig_r}, the best improvement is done by increasing both $L_{\pmb{x}}$ and $L_{\overline{\pmb{x}}}$.

%\todo[author=Adam, size=\footnotesize]{I would avoid color figures}

\begin{figure}[htb]
\centering
\begin{tikzpicture}
\draw [pattern=north west lines, pattern color=black] (0,0) rectangle (0.5,2);
\draw [pattern= grid, pattern color=black] (0.75,0) rectangle (1.25,2.5);
\draw [pattern= grid, pattern color=black] (1.5,0) rectangle (2,2.5);
\draw [pattern= grid, pattern color=black] (2.25,0) rectangle (2.75,2.5);
\draw [pattern= north west lines, pattern color=black] (3,0) rectangle (3.5,4);

\draw [color=black] (0,0) rectangle (0.5,2);
\draw [color=black] (0.75,0) rectangle (1.25,2.8);
\draw [color=black] (1.5,0) rectangle (2,2.7);
\draw [color=black] (2.25,0) rectangle (2.75,3);
\draw [color=black] (3,0) rectangle (3.5,4.2);

\draw [pattern= grid, pattern color=black] (4,0) rectangle (4.5,1);
\draw [pattern= grid, pattern color=black] (4.75,0) rectangle (5.25,1);
\draw [pattern= grid, pattern color=black] (5.5,0) rectangle (6,1);
\draw [pattern= grid, pattern color=black] (6.25,0) rectangle (6.75,1);
\draw [pattern= north west lines, pattern color=black] (7,0) rectangle (7.5,4);

\draw [color=black] (4,0) rectangle (4.5,2.3);
\draw [color=black] (4.75,0) rectangle (5.25,2.5);
\draw [color=black] (5.5,0) rectangle (6,2.6);
\draw [color=black] (6.25,0) rectangle (6.75,3);
\draw [color=black] (7,0) rectangle (7.5,4.3);

\draw [thick] (-0.5,2.5) -- (3.5,2.5);
\draw [thick] (4,1) -- (8,1);

\node at (-1,2.5) {$L_{\pmb{x}}$} ;
\node at (8.5,1) {$L_{\overline{\pmb{x}}}$};

\node at (2,-0.5) {$R_{\pmb{x}} =\frac{2}{3}$} ;
\node at (6,-0.5) {$R_{\overline{\pmb{x}}} =\frac{2}{4}$} ;

\end{tikzpicture}
\caption{Since $p=5$ and $k=2$, at least $3$ items must be chosen from $X_{\pmb{x}}$. Hence, the ratio $R_{\pmb{x}}$ is equal to $\frac{2}{3}$ and the algorithm increases $L_{\pmb{x}}$.}
\label{fig_r_x}
\end{figure}
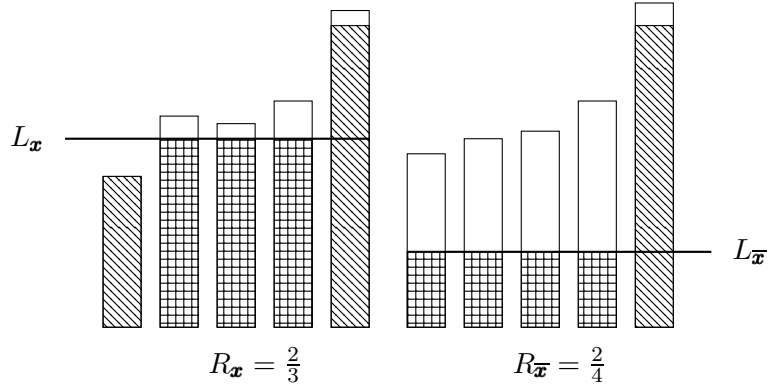

%\begin{figure} [htb]
%\centering
%\includegraphics[width=.6\textwidth]{fig_r_x.png}
%\caption{Since $p=5$ and $k=2$ at least $3$ items must be chosen from $X_{\pmb{x}}$. Hence, the ratio $R_{\pmb{x}}$ is equal to $\frac{2}{3}$ and the algorithm will increase $L_{\pmb{x}}$.}
%\label{fig_r_x}
%\end{figure}

\begin{figure}[htb]
\centering
\begin{tikzpicture}
\draw [pattern=north west lines, pattern color=black] (0,0) rectangle (0.5,2);
\draw [pattern= grid, pattern color=black] (0.75,0) rectangle (1.25,2.5);
\draw [pattern= grid, pattern color=black] (1.5,0) rectangle (2,2.5);
\draw [pattern= grid, pattern color=black] (2.25,0) rectangle (2.75,2.5);
\draw [pattern= north west lines, pattern color=black] (3,0) rectangle (3.5,4);

\draw [color=black] (0,0) rectangle (0.5,2);
\draw [color=black] (0.75,0) rectangle (1.25,2.8);
\draw [color=black] (1.5,0) rectangle (2,2.7);
\draw [color=black] (2.25,0) rectangle (2.75,3);
\draw [color=black] (3,0) rectangle (3.5,4.2);

\draw [pattern= grid, pattern color=black] (4,0) rectangle (4.5,1);
\draw [pattern= north west lines, pattern color=black] (4.75,0) rectangle (5.25,1.5);
\draw [pattern= north west lines, pattern color=black] (5.5,0) rectangle (6,2);
\draw [pattern= north west lines, pattern color=black] (6.25,0) rectangle (6.75,2.5);
\draw [pattern= north west lines, pattern color=black] (7,0) rectangle (7.5,4);

\draw [color=black] (4,0) rectangle (4.5,2.3);
\draw [color=black] (4.75,0) rectangle (5.25,2.5);
\draw [color=black] (5.5,0) rectangle (6,2.6);
\draw [color=black] (6.25,0) rectangle (6.75,3);
\draw [color=black] (7,0) rectangle (7.5,4.3);

\draw [thick] (-0.5,2.5) -- (3.5,2.5);
\draw [thick] (4,1) -- (8,1);

\node at (-1,2.5) {$L_{\pmb{x}}$} ;
\node at (8.5,1) {$L_{\overline{\pmb{x}}}$};

\node at (2,-0.5) {$R_{\pmb{x}} =\frac{2}{3}$} ;
\node at (6,-0.5) {$R_{\overline{\pmb{x}}} =1$} ;

\end{tikzpicture}
\caption{The ratio $R_{\overline{\pmb{x}}}$ is equal to $1$, hence it is optimal to increase $L_{\overline{\pmb{x}}}$.}
\label{fig_r_y}
\end{figure}
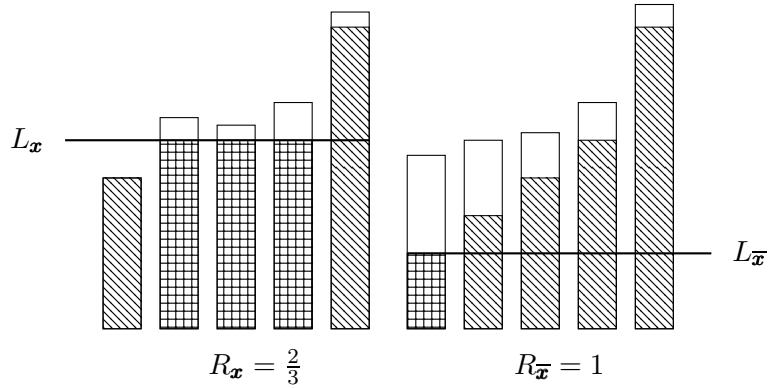

%\begin{figure}[htb]
%\centering
%\includegraphics[width=.6\textwidth]{fig_r_y.png}
%\caption{The ratio $R_{\overline{\pmb{x}}}$ is equal to $1$, hence it is optimal to increase $L_{\overline{\pmb{x}}}$.}
%\label{fig_r_y}
%\end{figure}

\begin{figure}[htb]
\centering
\begin{tikzpicture}
\draw [pattern=north west lines, pattern color=black] (0,0) rectangle (0.5,2);
\draw [pattern= grid, pattern color=black] (0.75,0) rectangle (1.25,2.5);
\draw [pattern= grid, pattern color=black] (1.5,0) rectangle (2,2.5);
\draw [pattern= grid, pattern color=black] (2.25,0) rectangle (2.75,2.5);
\draw [pattern= north west lines, pattern color=black] (3,0) rectangle (3.5,4);

\draw [color=black] (0,0) rectangle (0.5,2);
\draw [color=black] (0.75,0) rectangle (1.25,2.8);
\draw [color=black] (1.5,0) rectangle (2,2.7);
\draw [color=black] (2.25,0) rectangle (2.75,3);
\draw [color=black] (3,0) rectangle (3.5,4.2);

\draw [pattern= grid, pattern color=black] (4,0) rectangle (4.5,2.5);
\draw [pattern= grid, pattern color=black] (4.75,0) rectangle (5.25,2.5);
\draw [pattern= grid, pattern color=black] (5.5,0) rectangle (6,2.5);
\draw [pattern= grid, pattern color=black] (6.25,0) rectangle (6.75,2.5);
\draw [pattern= north west lines, pattern color=black] (7,0) rectangle (7.5,4);

\draw [color=black] (4,0) rectangle (4.5,3.5);
\draw [color=black] (4.75,0) rectangle (5.25,2.7);
\draw [color=black] (5.5,0) rectangle (6,3.8);
\draw [color=black] (6.25,0) rectangle (6.75,3);
\draw [color=black] (7,0) rectangle (7.5,4.3);

\draw [thick] (-0.5,2.5) -- (3.5,2.5);
\draw [thick] (4,2.5) -- (8,2.5);

\node at (-1,2.5) {$L_{\pmb{x}}$} ;
\node at (8.5,2.5) {$L_{\overline{\pmb{x}}}$};

\node at (2,-0.5) {$R_{\pmb{x}} =\frac{2}{4}$} ;
\node at (4,-0.5) {$R =\frac{4}{7}$} ;
\node at (6,-0.5) {$R_{\overline{\pmb{x}}} = \frac{0}{4}$} ;

\end{tikzpicture}
\caption{Recall that no items need to be chosen from $\overline{X}_{\pmb{x}}$. Hence, the ratio $R_{\overline{\pmb{x}}}=0$. Note, the importance of tie-breaking rule~1 (which is used in the definition of the set of selected items $I_{\overline{\pmb{x}}}(\pmb{c})$). In fact, the global ratio $R> R_{\pmb{x}}$, hence, it is optimal to increase $L_{\pmb{x}}$ and $L_{\overline{\pmb{x}}}$ in parallel.}
\label{fig_r}
\end{figure}
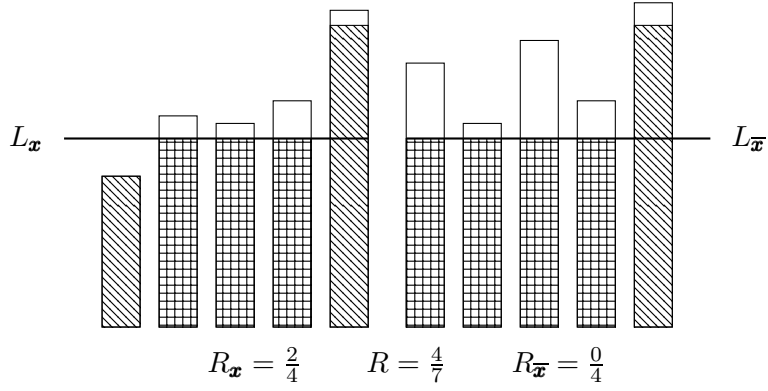

%\begin{figure}[htb]
%\centering
%\includegraphics[width=.6\textwidth]{fig_r.png}
%\caption{Recall that no items need to be chosen from $\overline{X}_{\pmb{x}}$. Hence, the ratio $R_{\overline{\pmb{x}}}=0$. Note, the importance of tie-breaking rule~1 (which is used in the definition of the set of selected items $I_{\overline{\pmb{x}}}(\pmb{c})$). In fact, the global ratio $R> R_{\pmb{x}}$, hence, it is optimal to increase $L_{\pmb{x}}$ and $L_{\overline{\pmb{x}}}$ in parallel.}
%\label{fig_r}
%\end{figure}

%Before we can analyze the runtime of the algorithm we have to further define the amount of budget it invests in each step to update the cost vector. The levels $L_{\pmb{x}}$ or $L_{\overline{\pmb{x}}}$ are increased in each step until either $R_{\pmb{x}}$ or $R_{\overline{\pmb{x}}}$ changes or the complete budget is spent (this defines the last step of the algorithm). It is important to realize that if we increase $L_{\pmb{x}}$ or $L_{\overline{\pmb{x}}}$, $R_{\pmb{x}}$ or $R_{\overline{\pmb{x}}}$ can only change at points contained in $\{c_1,\dots,c_n\} \cup \{c_1+d_1,\dots,c_n+d_n\}$. Hence, we know that the algorithm will terminate after at most $O(n)$ many update steps.

After the costs and levels are updated to $\pmb{c}^1$ and $L_{\pmb{x}}^1,L_{\overline{\pmb{x}}}^1$, the algorithm updates $R_{\pmb{x}}:=R_{\pmb{x}}(L_{\pmb{x}}^1,L_{\overline{\pmb{x}}}^1)$, $R_{\overline{\pmb{x}}}:=R_{\overline{\pmb{x}}}(L_{\pmb{x}}^1,L_{\overline{\pmb{x}}}^1)$, and $R:=R(L_{\pmb{x}}^1,L_{\overline{\pmb{x}}}^1)$ an repeats the process until the complete budget $\Gamma$ is spent or each item is at its maximum cost level.

Since the efficiency ratios can only change at interesting levels or if $L_{\pmb{x}}$ and $L_{\overline{\pmb{x}}}$ become equal, we conclude that the algorithm can perform at most $O(n)$ many steps before either the complete budget $\Gamma$ is spent or the cost of each item $c_i$ is raised to its maximum cost~$\cu_i$. We discuss at the end of this section how to efficiently implement the algorithm to have an amortized complexity of $O(\log(n))$ in each step. This gives the claimed $O(n\log(n))$ time complexity for the complete algorithm.

%Using efficient data structures \todo[author=Marc Pawel,  size=\footnotesize]{like what?} \todo[author=Adam, size=\footnotesize]{This is crucial. The running time $O(n \log n)$ is not obvious. I can see the running time $O(pn)$} to store the sets $A_{\pmb{x}}, A_{\overline{\pmb{x}}},I_{\pmb{x}},$ and $I_{\overline{\pmb{x}}}$ the algorithm can be implement to have an amortized running time of $O(\operatorname{log} n)$ for each update step, which leads to an overall running time of $O(n \operatorname{log}n)$. 

\noindent
The following lemma is important for the analysis of the algorithm.

%
%\begin{observation}
%Assume we fix a level $L_{\pmb{x}}$. The level $L_{\overline{\pmb{x}}}$ to which solution~$\pmb{c}^0$ can be enhanced, such that it becomes $(L_{\pmb{x}},L_{\overline{\pmb{x}}})-$compatible is uniquely determined.  Further, if $L_{\pmb{x}}$ increases $L_{\overline{\pmb{x}}}$ needs to decrease and vice versa. 
%\label{obs_unique_relation}
%\end{observation}
%
%\begin{observation}
%Assume we have a solution $\pmb{c}$ which is $(L_{\pmb{x}},L_{\overline{\pmb{x}}})-$compatible. If we invest budget to increase $L_{\overline{\pmb{x}}}$ and keep $L_{\pmb{x}}$ fix, ratio $R_{\pmb{x}}$ is non decreasing and $R_{\overline{\pmb{x}}}$ is non increasing. Contrary, if we decrease the investment to decrease $L_{\overline{\pmb{x}}}$ and keep $L_{\pmb{x}}$ fix, ratio $R_{\pmb{x}}$ is non increasing and $R_{\overline{\pmb{x}}}$ is non decreasing. The statements are analogous for the case that $L_{\overline{\pmb{x}}}$ is kept fix.
%\label{obs_inc_dec_ratio}
%\end{observation}

\begin{lemma}
For $L_{\pmb{x}} \leq L_{\pmb{x}}'$ and $L_{\overline{\pmb{x}}} \leq L_{\overline{\pmb{x}}}'$, we have that 
\begin{itemize}
\item[(i)] $R_{\pmb{x}}(L_{\pmb{x}},L_{\overline{\pmb{x}}}) \geq R_{\pmb{x}}(L_{\pmb{x}}',L_{\overline{\pmb{x}}})$
\item[(ii)] $R_{\overline{\pmb{x}}}(L_{\pmb{x}},L_{\overline{\pmb{x}}}) \leq R_{\overline{\pmb{x}}}(L_{\pmb{x}}',L_{\overline{\pmb{x}}})$
\item[(iii)] $R_{\overline{\pmb{x}}}(L_{\pmb{x}},L_{\overline{\pmb{x}}}) \geq R_{\overline{\pmb{x}}}(L_{\pmb{x}},L_{\overline{\pmb{x}}}')$
\item[(iv)] $R_{\pmb{x}}(L_{\pmb{x}},L_{\overline{\pmb{x}}}) \leq R_{\pmb{x}}(L_{\pmb{x}},L_{\overline{\pmb{x}}}')$
\end{itemize}
\label{lem_inc_dec_ratio}
\end{lemma}
\begin{proof}
Point $(ii)$ follows from the facts that $A_{\overline{\pmb{x}}}(L_{\pmb{x}},L_{\overline{\pmb{x}}})=A_{\overline{\pmb{x}}}(L_{\pmb{x}}',L_{\overline{\pmb{x}}})$ and $I_{\overline{\pmb{x}}}(L_{\pmb{x}},L_{\overline{\pmb{x}}}) \subseteq I_{\overline{\pmb{x}}}(L_{\pmb{x}}',L_{\overline{\pmb{x}}})$ (point~$(iv)$ is analogous). Next consider point $(i)$ (point $(iii)$ is analogous). We assume that $R_{\pmb{x}}(L_{\pmb{x}},L_{\overline{\pmb{x}}})<1$, since otherwise the inequality is trivial. First, observe that $I_{\pmb{x}}(L_{\pmb{x}}',L_{\overline{\pmb{x}}}) \subseteq I_{\pmb{x}}(L_{\pmb{x}},L_{\overline{\pmb{x}}})$. Assume we start at $L_{\pmb{x}}$ and increase it stepwise until it reaches $L_{\pmb{x}}'$. During each step $R_{\pmb{x}}$ should not change. Three different causes may lead to a change of $R_{\pmb{x}}$. First, one item from $X_{\pmb{x}}$ might become an active item, which was before not an active item. This increases $|A_{\pmb{x}}(L_{\pmb{x}},L_{\overline{\pmb{x}}})|$ by one, whereas $|I_{\pmb{x}}(L_{\pmb{x}},L_{\overline{\pmb{x}}}) \cap A_{\pmb{x}}(L_{\pmb{x}},L_{\overline{\pmb{x}}})|$ stays constant (since $R_{\pmb{x}}(L_{\pmb{x}},L_{\overline{\pmb{x}}})<1$). Second, an item $i$ which was active may reach its upper bound $\cu_i$. If item $i$ is one of the selected items, we get that both $|A_{\pmb{x}}(L_{\pmb{x}},L_{\overline{\pmb{x}}})|$ and $|I_{\pmb{x}}(L_{\pmb{x}},L_{\overline{\pmb{x}}}) \cap A_{\pmb{x}}(L_{\pmb{x}},L_{\overline{\pmb{x}}})|$ are reduced by one. Note that if $i$ is not among the selected items, we must have that $|I_{\pmb{x}}(L_{\pmb{x}},L_{\overline{\pmb{x}}}) \cap A_{\pmb{x}}(L_{\pmb{x}},L_{\overline{\pmb{x}}})|=0$, and hence $R_{\pmb{x}}=0$ and can not decrease any further. Third, an item $i$ leaves the set of selected items, in this case $|I_{\pmb{x}}(L_{\pmb{x}},L_{\overline{\pmb{x}}}) \cap A_{\pmb{x}}(L_{\pmb{x}},L_{\overline{\pmb{x}}})|$ is reduced by one. Note that all three cases, lead to a reduction of $R_{\pmb{x}}$ (except for the case where $R_{\pmb{x}}=0$).
\end{proof}

\begin{lemma}
Solution $\tilde{\pmb{c}}$ computed by the algorithm is optimal.
\label{lem_sol_is_opt}
\end{lemma}
\begin{proof}
Denote by $\tilde{L}_{\pmb{x}}$ and $\tilde{L}_{\overline{\pmb{x}}}$ the corresponding levels produced by the algorithm. Let $\pmb{c}^*$ be the optimal solution of problem  \textsc{AREC}. Assume that $F(\tilde{\pmb{c}}) < F(\pmb{c}^*)$, i.e. $\tilde{\pmb{c}}$ is not an optimal solution. For a small $\epsilon>0$, we consider solution $\hat{\pmb{c}} := (1-\epsilon) \tilde{\pmb{c}} + \epsilon \pmb{c}^*$. Since $F$ is concave (Lemma~\ref{lem_f_concave}), it must hold that $F(\tilde{\pmb{c}}) < F(\hat{\pmb{c}})$. Due to Lemma~\ref{lem_transform_to_compatible}, we can transform solution $\hat{\pmb{c}}$ to a solution $\pmb{c}'$ which is $(L_{\pmb{x}}',L_{\overline{\pmb{x}}}')-$compatible and has no smaller objective value, i.e. $F(\hat{\pmb{c}}) \leq F(\pmb{c}')$. We consider three different cases: \\

\noindent
\textbf{Case~1:} $L_{\pmb{x}}' = \tilde{L}_{\pmb{x}}$ \\
We can assume without loss of generality that the complete budget $\Gamma$ is spent to increase $\underline{\pmb{c}}$ 
to $\tilde{\pmb{c}}$ respectively $\pmb{c}'$. Hence, we are able to conclude that $L_{\overline{\pmb{x}}}' = \tilde{L}_{\overline{\pmb{x}}}$. In consequence, solution $\pmb{c}'$ and $\tilde{\pmb{c}}$ are identical, contradicting that $F(\tilde{\pmb{c}}) < F(\pmb{c}')$.\\

\noindent
\textbf{Case~2:} $L_{\pmb{x}}' < \tilde{L}_{\pmb{x}}$ \\
With the same reasoning as in Case~1, we conclude that $L_{\overline{\pmb{x}}}' > \tilde{L}_{\overline{\pmb{x}}}$. To arrive at a contradiction we have to take into account the last step of the algorithm, where $L_{\pmb{x}}$ was increased to $\tilde{L}_{\pmb{x}}$ (note that if $L_{\pmb{x}}$ was never increased during the algorithm, it must be at its initial level and, hence, $L_{\pmb{x}}' < \tilde{L}_{\pmb{x}}$ cannot be true). We denote by $\pmb{c}^-$ the solution of the algorithm before this step. The corresponding levels of $\pmb{c}^-$ are denoted by $L_{\pmb{x}}^-$ and $L_{\overline{\pmb{x}}}^-$. We can assume that $L_{\pmb{x}}'$ is arbitrary close to $\tilde{L}_{\pmb{x}}$, since the $\epsilon$ that was used in the definition of $\hat{\pmb{c}}$ can be arbitrary small. So we get that $L_{\pmb{x}}^- < L_{\pmb{x}}' < \tilde{L}_{\pmb{x}}$ and $L_{\overline{\pmb{x}}}^- \leq \tilde{L}_{\overline{\pmb{x}}} < L_{\overline{\pmb{x}}}'$. Two different cases were possible for this step of the algorithm. \\

\noindent
\textbf{Case~2a:} $L_{\overline{\pmb{x}}}$ stayed constant during this step. \\
Note that solution $\pmb{c}'$ and solution $\tilde{\pmb{c}}$ can both be created from $\pmb{c}(L_{\pmb{x}}',\tilde{L}_{\overline{\pmb{x}}})$ by increasing $L_{\pmb{x}}$ or $L_{\overline{\pmb{x}}}$. Denote by $\Gamma_{\pmb{x}}$ the budget which is necessary to increase $L_{\pmb{x}}'$ to $\tilde{L}_{\pmb{x}}$ and by $\Gamma_{\overline{\pmb{x}}}$ the budget which is necessary to increase $\tilde{L}_{\pmb{y}}$ to $L_{\pmb{y}}'$. The objective values of both solutions are given by
\[F(L_{\pmb{x}}',L_{\overline{\pmb{x}}}') = F(L_{\pmb{x}}',\tilde{L}_{\overline{\pmb{x}}}) + \Gamma_{\overline{\pmb{x}}} R_{\overline{\pmb{x}}}(L_{\pmb{x}}',\tilde{L}_{\overline{\pmb{x}}}),\]
\[F(\tilde{L}_{\pmb{x}},\tilde{L}_{\overline{\pmb{x}}}) = F(L_{\pmb{x}}',\tilde{L}_{\overline{\pmb{x}}}) + \Gamma_{\pmb{x}} R_{\pmb{x}}(L_{\pmb{x}}',\tilde{L}_{\overline{\pmb{x}}}).\]
Observe that we must have $\Gamma_{\pmb{x}} = \Gamma_{\overline{\pmb{x}}}$ since we can assume without loss of generality that the complete budget is spent to create solution $\tilde{\pmb{c}}$ and solution $\pmb{c}'$. The following estimations show that $R_{\pmb{x}}(L_{\pmb{x}}',\tilde{L}_{\overline{\pmb{x}}}) \geq R_{\overline{\pmb{x}}}(L_{\pmb{x}}',\tilde{L}_{\overline{\pmb{x}}})$, which leads to the desired contradiction.

\begin{align*}
R_{\pmb{x}}(L_{\pmb{x}}',\tilde{L}_{\overline{\pmb{x}}}) &\geq R_{\pmb{x}}(L_{\pmb{x}}',L_{\overline{\pmb{x}}}^-) &&\text{Lemma~\ref{lem_inc_dec_ratio} } (iv) \\
&= R_{\pmb{x}}(L_{\pmb{x}}^-,L_{\overline{\pmb{x}}}^-) &&\text{Ratios are constant during one step of the algorithm}  \\
&\geq R_{\overline{\pmb{x}}}(L_{\pmb{x}}^-,L_{\overline{\pmb{x}}}^-) &&\text{Choice of the algorithm}  \\
&= R_{\overline{\pmb{x}}}(L_{\pmb{x}}',L_{\overline{\pmb{x}}}^-) &&\text{Ratios are constant during one step of the algorithm}  \\
&\geq R_{\overline{\pmb{x}}}(L_{\pmb{x}}',\tilde{L}_{\overline{\pmb{x}}}) &&\text{Lemma~\ref{lem_inc_dec_ratio} } (iii)
\end{align*}

\noindent
\textbf{Case~2b:} $L_{\overline{\pmb{x}}}$ was increased during this step. \\
In this case, we must have that $L_{\pmb{x}}^- = L_{\overline{\pmb{x}}}^-$. Instead of performing a full update step from $L_{\pmb{x}}^-$ to $\tilde{L}_{\pmb{x}}$, we perform only a partial step and increase $L_{\pmb{x}}^-$ and $L_{\overline{\pmb{x}}}^-$ to $L_{\pmb{x}}'$. After this partial step we arrive at  solution $\pmb{c}(L_{\pmb{x}}',L_{\pmb{x}}')$. Note that $\tilde{\pmb{c}}$ and $\pmb{c}'$ can both be created from $\pmb{c}(L_{\pmb{x}}',L_{\pmb{x}}')$ by increasing $L_{\pmb{x}}$ or $L_{\overline{\pmb{x}}}$. Denote by $\Gamma_{\pmb{x}}$ the budget, which is necessary to increase $L_{\pmb{x}}$ from $L_{\pmb{x}}'$ to $\tilde{L}_{\pmb{x}}$ and by $\Gamma_{\overline{\pmb{x}}}$ the budget which is necessary to increase $L_{\overline{\pmb{x}}}$ from $L_{\pmb{x}}'$ to $\tilde{L}_{\pmb{x}}$. After the last step of the algorithm where $L_{\pmb{x}}$ was increased to $L_{\overline{\pmb{x}}}$, the algorithm performs a sequence of $k$ updating moves increasing $L_{\overline{\pmb{x}}}$, denote by $\Gamma_{\overline{\pmb{x}}}^j$ the budget which is spent in the $j$th updating move and by $L_{\overline{\pmb{x}}}^j$ the level $L_{\overline{\pmb{x}}}$ before the $j$th updating move. Finally, denote by $\Gamma_{\overline{\pmb{x}}}'$ the budget which is necessary to update $L_{\overline{\pmb{x}}}$ from $\tilde{L}_{\overline{\pmb{x}}}$ to $L_{\pmb{y}}'$. This allows us to relate the objective values of $\tilde{\pmb{c}}$ and $\pmb{c}'$ to the objective value of $\pmb{c}(L_{\pmb{x}}',L_{\pmb{x}}')$.

\begin{align*}
F(L_{\pmb{x}}',L_{\overline{\pmb{x}}}') &= F(L_{\pmb{x}}',L_{\pmb{x}}') + \Gamma_{\overline{\pmb{x}}} R_{\overline{\pmb{x}}}(L_{\pmb{x}}',L_{\pmb{x}}') + \sum_{j=1}^k \Gamma_{\overline{\pmb{x}}}^j R_{\overline{\pmb{x}}}(L_{\pmb{x}}',L_{\overline{\pmb{x}}}^j) + \Gamma_{\overline{\pmb{x}}}' R_{\overline{\pmb{x}}}(L_{\pmb{x}}',\tilde{L}_{\overline{\pmb{x}}})  \\
F(\tilde{L}_{\pmb{x}},\tilde{L}_{\overline{\pmb{x}}}) &= F(L_{\pmb{x}}',L_{\pmb{x}}') + (\Gamma_{\pmb{x}} + \Gamma_{\overline{\pmb{x}}}) R(L_{\pmb{x}}',L_{\pmb{x}}')  + \sum_{j=1}^k \Gamma_{\overline{\pmb{x}}}^j R_{\overline{\pmb{x}}}(\tilde{L}_{\pmb{x}},L_{\overline{\pmb{x}}}^j).
\end{align*}

\noindent
Consider the following estimations for the different efficiency ratios.

\begin{align*}
R(L_{\pmb{x}}',L_{\pmb{x}}') &= R(L_{\pmb{x}}^-,L_{\pmb{x}}^-) &&\text{Ratios are constant during one step of the algorithm}  \\
&\geq R_{\overline{\pmb{x}}}(L_{\pmb{x}}^-,L_{\pmb{x}}^-)&&\text{Choice of the algorithm}  \\
&= R_{\overline{\pmb{x}}}(L_{\pmb{x}}',L_{\pmb{x}}') &&\text{Ratios are constant during one step of the algorithm}  \\
&\geq R_{\overline{\pmb{x}}}(L_{\pmb{x}}',\tilde{L}_{\overline{\pmb{x}}})  &&\text{Lemma~\ref{lem_inc_dec_ratio} } (iii) \\
R_{\overline{\pmb{x}}}(\tilde{L}_{\pmb{x}},L_{\overline{\pmb{x}}}^j) &\geq R_{\overline{\pmb{x}}}(L_{\pmb{x}}',L_{\overline{\pmb{x}}}^j)  &&\text{Lemma~\ref{lem_inc_dec_ratio} } (ii)
\end{align*}

Using these estimations it is straightforward to show that  $F(\tilde{L}_{\pmb{x}},\tilde{L}_{\overline{\pmb{x}}}) \geq F(L_{\pmb{x}}',L_{\overline{\pmb{x}}}')$, which leads to the desired contradiction. \\

\noindent
\textbf{Case~3:} $L_{\pmb{x}}' > \tilde{L}_{\pmb{x}}$ \\
This case is completely analogous to Case~2, just exchange $\pmb{x}$ and $\overline{\pmb{x}}$.\\

All cases lead to contradictions, which proves that the assumption $F(\tilde{\pmb{c}}) < F(\pmb{c}^*)$ is wrong. Hence, the solution provided by the algorithm is indeed optimal. 
\end{proof}

Lemma~\ref{lem_sol_is_opt} completes the analysis of the proposed $O(n\log (n))$ algorithm. We summarize the findings of this section in the following theorem.

\begin{theorem}
Problem  \textsc{AREC}  can be solved in $O(n\log (n))$.
\end{theorem}

Problem     \textsc{A2ST}    can be seen as a simpler version of problem  \textsc{AREC}. All items which were already selected in the first stage are removed from the instance. The remaining problem is a special instance of problem  \textsc{AREC}  (where $\overline{X}_{\pmb{x}}$ is empty).

\begin{corollary}
Problem \textsc{A2ST}  can be solved in $O(n\log(n))$.
\end{corollary}

We only sketch the idea how to efficiently implement the algorithm, since the exact details are cumbersome. For an efficient implementation of the algorithm we renounce an explicit representation of  $A_{\pmb{x}}, A_{\overline{\pmb{x}}},I_{\pmb{x}},$ and $I_{\overline{\pmb{x}}}$. Instead we store these sets only implicitly and compute the values which are important to compute the different efficiency ratios. To represent the set $A_{\pmb{x}}$ ($A_{\overline{\pmb{x}}}$ is analogous) we store two pointers representing the leftmost and rightmost element of the set (sorted with respect to the actual cost). If the level $L_{\pmb{x}}$ is raised such that one other item becomes active we increase the right pointer by one. If the cost of an active item is raised to its upper bound the item becomes inactive and we increase the left pointer by one. Since we have only an implicit representation of $A_{\pmb{x}}$, it is not trivial to efficiently check whether the cost of an item reaches its upper bound. To handle this problem we create a \emph{min-heap} which contains the upper bound cost of all active items. Note that each item becomes at most once active and at most once inactive. Inserting and removing an item into the heap requires $O(\log(n))$
(see, e.g.,~\cite{CO90}). Hence, building and maintaining the heap requires overall $O(n \log(n))$ time.

The set of selected items $I_{\pmb{x}}$ and $I_{\overline{\pmb{x}}}$ is represented by two pointers, which indicate at which positions the levels $L_{\pmb{x}}$ and $L_{\overline{\pmb{x}}}$ lie with respect to the actual cost of the items in $X_{\overline{\pmb{x}}}$ and $X_{\pmb{x}}$. Further we keep track of all items which became inactive since they have reached its cost upper bound. Note that these items are certainly part of the selected items. This information allows us to compute $|I_{\pmb{x}}(L_{\pmb{x}},L_{\overline{\pmb{x}}}) \cap A_{\pmb{x}}|$ and $|I_{\overline{\pmb{x}}}(L_{\pmb{x}},L_{\overline{\pmb{x}}}) \cap A_{\overline{\pmb{x}}}(L_{\pmb{x}},L_{\overline{\pmb{x}}})|$ in constant time in each step of the algorithm. Updating the pointers requires overall $O(n)$ time.

For choosing the correct update step length for $L_{\pmb{x}}$ ($L_{\overline{\pmb{x}}}$ is analogous), we take the minimum of four values. First, the cost of the item of $X_{\pmb{x}}$, which is the next item that becomes active, if $L_{\pmb{x}}$ is increased. Second, the minimum of the upper bound cost of all active items of $A_{\pmb{x}}$ (provided by the min-heap). Third, the level $L_{\overline{\pmb{x}}}$. Fourth, the cost of the cheapest item of $X_{\overline{\pmb{x}}}$ which has cost strictly higher than $L_{\pmb{x}}$. Note that we can obtain each value in constant time. Using these update step sizes, it is ensured that after each step all values required to compute the efficiency ratios can be updated efficiently.

\end{document}